\documentclass[11pt]{article}

\usepackage{amssymb, graphicx}
\usepackage[bookmarksnumbered, colorlinks, plainpages]{hyperref}
\hypersetup{colorlinks=true,linkcolor=red, anchorcolor=green, citecolor=cyan, urlcolor=red, filecolor=magenta, pdftoolbar=true}
\usepackage[centertags]{amsmath}
\usepackage{mathrsfs}
\usepackage{amsthm}
\allowdisplaybreaks[4]
\usepackage[top=2cm,bottom=2cm,left=3cm,right=3cm,includehead,includefoot]{geometry}
\usepackage{indentfirst}
\usepackage{cite}
 \newtheorem{thm}{Theorem}[section]
 
 \newtheorem{lem}[thm]{Lemma}
 
 \theoremstyle{definition}
 
 \theoremstyle{remark}

 \numberwithin{equation}{section}

\begin{document}
\title{{ Compact linear combination of composition operators on Bergman-Orlicz spaces \thanks{ This work is supported by the National Natural Science Foundation of China (12171138,12371129)} }
\author{Yucheng Li$^{a}$,\, Zhiyu Wang$^{a}$\thanks{Corresponding author: wangzhiyumath@163.com} and Liankuo Zhao$^{b}$ \\
{\small \texttt{liyucheng@hebtu.edu.cn,\,\,wangzhiyumath@163.com,\,\,lkzhao@sxnu.edu.cn}}\\
  \small {$^{a}$\,\,\,\,School of Mathematical Sciences, Hebei Normal University, Shijiazhuang 050024,  China  }\\
  \small {Hebei Key Laboratory of Computational Mathematics and Applications } \\
\small {$^{b}$\,\,\,\,School of Mathematical Sciences, Shanxi  Normal University, Taiyuan 030031,  China}}}

\date{}
\maketitle

\textbf{Abstract}\ Let $\Psi$ be an Orlicz function, and let $T$ denote any finite linear combination of composition operators. Motivated by the work of Choe, Koo, and Wang on the compactness of linear combinations of composition operators on weighted Bergman spaces, we completely characterize the compactness of $T$  on the Bergman-Orlicz space $A^{\Psi}(\mathbb D)$. Unlike the classical weighted Bergman space setting, the precise analysis on $A^{\Psi}(\mathbb D)$ depends on the growth behavior of $\Psi$ and the associated Luxemburg norm. Hence, the power-type estimates available in weighted Bergman spaces cannot be applied directly. To overcome this difficulty, under suitable growth assumptions on $\Psi$ and using the Julia-Carathéodory type function theory and Carleson type measure technique, respectively, we establish two necessary and sufficient compactness characterizations for $T$ on $A^{\Psi}(\mathbb D)$.



\section{Introduction}

Let $\mathbb{D}=\{z:|z|<1\}$ be the open unit disk of the complex plane $\mathbb{C}$, and $\mathbb{T}=\{z:|z|=1\}$ be the unit circle.
Let $H(\mathbb{D})$ denote the space of all holomorphic functions on $\mathbb{D}$. We use $S(\mathbb{D})$ to denote the set of all holomorphic self-maps on $\mathbb{D}$, for $\varphi\in S(\mathbb{D})$, the composition operator $C_{\varphi}$ is defined by
 $$C_{\varphi}f=f\circ\varphi,   \qquad f\in H(\mathbb{D}).$$
 Composition operators between various spaces of holomorphic functions have been extensively studied in the literature (see \cite{BWY,CM,Do,GW,KW,KL,Sh,WX,ZC}).
By Littlewood's subordination principle, every composition operator is bounded on both the classical Hardy spaces and the weighted Bergman spaces.
More recently, Sharma \cite{SS} pointed out that all the composition operators are also bounded on Bergman-Orlicz spaces for every Orlicz functions.
Recall that a function $\Psi:[0,\infty]\to[0,\infty]$ is called an Orlicz functions if it is convex, non-decreasing and satisfies $\Psi(0)=0$ and $\Psi(\infty)=\infty$.
Furthermore, in this paper, we also suppose that $\Psi$ satisfies the following properties: it is continuous at $0$, strictly convex (strictly increasing), and
$$ \frac{\Psi(t)}{t}\longrightarrow \infty  \,\,\,\,\,\, as \,\,\,\,\, t\rightarrow\infty. $$

The Bergman-Orlicz space $A^{\Psi}(\mathbb{D})$ is defined as the set of all holomorphic functions $f$ on $\mathbb{D}$ for which there is a constant $C>0$ such that
$$\int _ {\mathbb{D}} \Psi \bigg( \frac { | f (z) | } { C } \bigg)  {\rm d}A (z) < + \infty,$$
where $ {\rm d}A(z)=\frac{1}{\pi}{\rm d}x{\rm d}y$ is the normalized area measure on $\mathbb{D}$. For $f\in A^{\Psi}(\mathbb{D})$, the Luxemburg norm of $f$ is defined as:
$$\|f\|_{A^{\Psi}(\mathbb{D})}=\inf\bigg\{C>0: \int _ { \mathbb{D} } \Psi \bigg( \frac { | f (z) | } { C } \bigg)  {\rm d} A (z) \leq1\bigg\}.$$
Equipped with this Luxemburg norm,
$A^{\Psi}(\mathbb{D})$ is a Banach space {\rm (see \cite{Ch1,OF})}. For $\Psi(t)=t^{p}(1< p<\infty)$, we obtain the classical Bergman space $A^{p}(\mathbb{D})$. Compared with classical Bergman spaces, Bergman-Orlicz spaces depend on the growth behavior of the Orlicz function. We therefore introduce the following growth conditions.

Let $\Psi$ be an Orlicz function. We say that $\Psi$ satisfies the $ \Delta_{2}$-condition if there exists a constant $C>1$ such that
\begin{align}\label{eq: condition}
  \Psi (2t) & \leq C \Psi (t)\ \ \mbox{for}\ \ t>0.
\end{align}

 An Orlicz function $ \Psi$  is called positive upper type  $p\, (p>0)$(see \cite{YLK}), if there exists a constant $C>0$  such that
\begin{align}\label{eq: positive upper type}
  \Psi(st) & \leq Cs^{p}\Psi(t)\ \ \mbox{for all}\ \  t>0\ \mbox{and}\  s\geq1.
\end{align}
By a straightforward calculation, it is easy to see the equivalence between \eqref{eq: condition} and \eqref{eq: positive upper type}.
We have the following Remark about  $\Delta_{2}$-condition (see Remarks 2(a) following Theorem 4.11 in {\rm\cite{LLQR2}}).

\noindent\textbf{Remark}. Let $\Psi$ be an Orlicz function satisfying the $ \Delta_{2}$-condition. Then for every $B>0$, there exists a constant $C=C_{B}>0$ such that
\begin{align}\label{eq: remark}
   \frac{1}{C}\Psi(t) & \leq\Psi(Bt)\leq C \Psi(t)\ \ \mbox{ for all}\ \ t>0.
\end{align}

In 2006, P. Lefèvre, D. Li, H. Queffélec, and L. Rodríguez-Piazza have undertook a systematic investigation of composition operators acting on both Hardy-Orlicz spaces $H^{\Psi}(\mathbb{D})$ and Bergman-Orlicz spaces $A^{\Psi}(\mathbb{D})$ (see \cite{LLQR,LLQR2,LLQR3,LLQR1}). Subsequently, composition operators on Bergman-Orlicz spaces and Hardy-Orlicz spaces have been extensively studied in the literature ( see \cite{Ch,Ch2,JC,LLQR1,LLQR2,Li,SS,SU}).
Motivated by the work of Choe, Koo and Wang \cite{CKW}, in this paper, we investigate the compact linear combination of composition operators on the setting of Bergman-Orlicz spaces.

MacCluer and Shapiro \cite{MS} proved the Julia-Carathéodory type condition
\begin{align}\label{eq: Julia-Caratheodory type}
R _ { \varphi } (z)  = \frac { 1 - | z | ^{2} } { 1 - | \varphi (z) | ^{2} } \rightarrow 0  \ \text{ as }\  | z | \to 1
\end{align}
is a necessary and sufficient condition for the compactness of composition operators on the weighted Bergman spaces. Furthermore,
 in 2005, Moorhouse \cite{Mo} first characterized the compactness of differences of composition operators on $A_{\alpha}^{2}(\mathbb{D})$ by the
Julia-Carathéodory type condition
\begin{align}\label{eq: Julia-Caratheodory type difference}
M _ { \varphi , \psi } (z)  = [ R _ { \varphi }(z) + R _ { \psi } (z) ] \rho _ { \varphi , \psi } (z) \to 0  \qquad    as  \,\,\,\,\,    | z | \to 1 ,
\end{align}
where $\rho$ denotes the pseudohyperbolic distance.

As is well known, Carleson measures serve as a crucial tool in characterizing the boundedness and compactness of composition operators. The connection between composition operators and Carleson measures comes from a change-of-variable formula (see \eqref{eq: variable-change formula} for the Bergman-Orlicz space).
As the change-of-variable formula cannot be applied to the difference of two composition operators, Koo and Wang \cite{KW1}
overcame this difficulty by introducing the novel concept of joint Carleson measure. They established a Carleson type measure characterization for the compact difference of composition operators acting on the weighted Bergman spaces over the unit ball of $\mathbb{C}^{n}$. Along the same line of the study on differences of composition operators, considerable effort has been devoted to the research of linear combinations of composition operators.
Kriete and Moorhouse \cite{KM} first investigated the linear combination of composition operators on the weighted Bergman  space $A_{\alpha}^{2}(\mathbb{D})$. Later, Choe, Koo and Wang \cite{CKW}  characterized  the compact linear combination of composition operators with arbitrary symbols on $A_{\alpha}^{p}(\mathbb{D})$ by Julia-Carathéodory type characterization and Carleson type measure. Further results on the characterization of composition operators in terms of Carleson measures can be found in \cite{Ch1,CS,Ji,LLQR2,OF}.


To state our main results, we introduce some notations.
Let $N \geq 2$ be an arbitrary integer, and let $\mathscr{P}_N$ denote the group of all permutations of the index set $$\Lambda _ { N } = \{ 1 , 2 , \ldots , N \} .$$
Assume that the coefficients
$$a _ { 1 } , \ldots , a _ { N } \in \mathbb { C } \backslash \{ 0 \}$$
and the symbol functions
$$\varphi _ { 1 } , \varphi _ { 2 } , \ldots, \varphi _ { N } \in S(\mathbb{D}).$$
Let
$$T _ { i }  = C _ { \varphi _ { i } }  ,   \,\,\,     T _ { i , k }  = T _ { i } - T _ { k}$$
for $i,k=1,\ldots,N$. For simplicity, denote the linear combination operators by
$$T = \sum _ { i= 1 } ^{N} a _ { i } T _ { i }.$$
We identify each $\eta \in \mathscr{P}_N$ with the ordered $N$-tuple $(\eta_1, \ldots, \eta_N)$, where $\eta_i = \eta(i)$ and write $$\eta = ( \eta _ { 1 } , \ldots , \eta _ { N } ) .$$
For $\eta \in \mathscr{P}_N$, let
$$s _ { i } ^{\eta} = s _ { i } ^{\eta} ( a _ { 1 } , \cdot\cdot\cdot , a _ { N } )  = \sum _ { k = 1 } ^{i} a _ { \eta _ { k } }$$
for all $i \in \Lambda _ { N }$. The operator $T$  can be rewritten as
\begin{align}\label{eq: operator-T}
T = \sum _ { i = 1 } ^{N - 1} s _ { i } ^{\eta} T _ { \eta _ { i } , \eta _ { i + 1 } } + s _ { N } ^{\eta} T _ { \eta _ { N } }
\end{align}
for each $ \eta\in \mathscr{P}_N$.
Furthermore, let
$$R _ { i }  = R _ { \varphi _ { i } } , \,\,\,   M _ { i , k }  = M _ { \varphi _ { i } , \varphi _ { k } } ,  \,\,\,   \rho _ { i , k }  = \rho _ { \varphi _ { i } , \varphi _ { k } }$$
for $i,k\in\Lambda_N$ and $\varphi _ { i } , \varphi _ { k}\in S(\mathbb{D})$.
By the Schwarz-Pick lemma, each $R_i$ is bounded on $\mathbb{D}$. Consequently, each $M_{i,k}$ is also bounded on $\mathbb{D}$.

Define
\begin{align}\label{eq: define-Q}
Q _ { \eta}  = \sum _ { i = 1 } ^{N - 1} | s _ { i } ^{\eta} |  M _ { \eta _ { i } , \eta _ { i + 1 } } + | s _ { N } ^{\eta} |  R _ { \eta _ { N } }
\end{align}
for $\eta\in\mathscr{P}_{N}$, and set
\begin{align}\label{eq: Q}
Q  = \prod _ { \eta \in \mathscr { P } _ { N } } Q _ { \eta }.
\end{align}
Note that all functions  $Q_{\eta}$ are bounded on $\mathbb{D}$, we write
\begin{align}\label{eq: define-G}
G _ { \eta }  = \left\{ z \in \mathbb { D } : Q _ { \eta } (z) = \min _ { \tau \in \mathscr { P } _ { N } } Q _ { \tau } (z) \right\}
\end{align}
for each $\eta \in \mathscr{P}_N$.

The following theorem is our first main result: Characterizing compact linear combination of composition operators via Julia-Carathéodory type conditions. In what follows, $u_{z}$ is the function defined in Lemma \ref{le: test-function}.
\begin{thm}\label{th: Julia-Caratheodory-type}
Let $\Psi$ be an Orlicz function satisfying the $\Delta_{2}$-condition. Suppose that $\varphi_i\in S(\mathbb{D})$ and $a_i \in \mathbb{C} \backslash \{0\}$ for each $i\in\Lambda_N$. Then the following statements are equivalent:

{\rm (a)} $T$ is compact on $ A^{\Psi}(\mathbb{D})$;

{\rm (b)}  $\lim\limits_{|z|\to1}\frac{\|Tu_{z}\|_{A^{\Psi}(\mathbb{D})}}{\|u_{z}\|_{ A^{\Psi}(\mathbb{D}) }}=0$;

{\rm (c)} $\lim\limits_{|z|\to1}Q(z)=0$.
\end{thm}

Our second main result provides a Carleson measure characterization. We now introduce Carleson measure on the Bergman-Orlicz spaces.

For $\xi \in \mathbb{T}$ and $h \in (0,1)$, the Carleson window  $W ( \xi , h ) $ is defined by
$$W ( \xi , h ) = \{ z \in \mathbb { D }   ; | z |>1-h  ,  | \arg ( z \bar { \xi } ) | < h \} .$$
Let $\mu$ be a positive Borel measure on $\mathbb{D}$. We say that $\mu$ is a Carleson measure on $\mathbb{D}$
if there exists a constant $C > 0$ such that
$$\mu \big ( W ( \xi , h ) \big ) \leq Ch^{2}.$$
We say that $\mu$  is a vanishing Carleson measure on $\mathbb{D}$ if
$$\lim_{h\rightarrow0}\sup_{\xi\in\mathbb{T}}\frac{ \mu \big  (W ( \xi , h ) \big) }{h^{2}}=0 .$$

Assume that $\Psi$ be an Orlicz function satisfying the $\Delta_{2}$-condition, by Remark 2(a) following Theorem 4.11 in \cite{LLQR2}, for every $B>0$, we have
$$\frac{1}{\Psi(B\Psi^{-1}(1/h^{2}))}\approx h^{2}.$$
Therefore, applying Theorem 2.5 and Theorem 2.9 in \cite{Ch1} to the case $N=1$ and $\alpha=0$, we obtain the following characterization:
Identity embedding $I_{\mu} :A^{\Psi}(\mathbb{D})\rightarrow L^{\Psi}(\mu)$ is bounded if and only if
 $\mu$ is a Carleson measure.
 Identity embedding $I_{\mu} :A^{\Psi}(\mathbb{D})\rightarrow L^{\Psi}(\mu)$  is compact if and only if
$\mu$  is a vanishing Carleson measure. In particular, the following theorem will be used to prove our second main result.

\vskip.2cm

\noindent\textbf{Theorem A.}\label{th: pullback-measure}(\cite[Theorem 3.3]{Ch1})
{\it Let $\Psi$ be an Orlicz function satisfying the $\Delta_{2}$-condition, $\varphi\in S(\mathbb{D})$, $\mu$ be a positive  Borel measure on $\mathbb{D}$. Then}

{\rm (1)}  {\it $C_{\varphi}$ is bounded on $A^{\Psi}(\mathbb{D})$ if and only if $ A  \circ \varphi ^{- 1}$ is a Carleson measure.}

{\rm (2)} {\it $C_{\varphi}$ is compact on $A^{\Psi}(\mathbb{D})$ if and only if $ A  \circ \varphi ^{- 1}$ is a vanishing Carleson measure.}

\vskip.2cm

Given $\eta\in \mathscr{P}_N$, for any Borel set $E \subset \mathbb{D}$, we define the joint pullback measures $\mu_{\Psi,\eta_i} $ and $\nu_{\Psi,\eta_i} $ as follows,
$$\mu_{\Psi,\eta_i}(E)  = \int  _{\varphi _ { \eta _ { i } } ^{- 1} (E)  \cap G _ { \eta } } \Psi(M_{\eta _ { i },\eta _ { i+1 }}){\rm d} A +\int  _{\varphi _ { \eta _ { i+1 } } ^{- 1} (E)  \cap G _ { \eta } }\Psi( M_{\eta _ { i },\eta _ { i+1 }}) {\rm d} A, $$
$$\nu_{\Psi,\eta_i} (E)  = \int  _{\varphi _ { \eta _ { i } } ^{- 1} (E)  \cap G _ { \eta } } \Psi(\rho_{\eta _ { i },\eta _ { i+1 }}) {\rm d} A +\int  _{\varphi _ { \eta _ { i+1 } } ^{- 1} (E)  \cap G _ { \eta } } \Psi(\rho_{\eta _ { i },\eta _ { i+1 }}) {\rm d} A $$
for $1 \leq i < N$ and
$$\mu _ { \Psi,\eta _ { N } } (E)  = \nu _ { \Psi,\eta _ { N } } (E)  = \int _ {   \varphi _ { \eta _ { N } } ^{- 1} (E) \cap G _ { \eta }} \Psi( R _ { \eta _ { N } } )  {\rm d} A $$
for $ i = N$. Set
\begin{align}\label{eq: joint pullback measures}
\mu _{\Psi} = \sum _ { \eta \in \mathscr { P } _ { N } } \mu _{\Psi,\eta},  \quad \text {where} \quad \mu _{\Psi,\eta}  = \sum _ { i = 1 } ^{N} | s _ { i} ^{\eta} |  \mu _ { \Psi,\eta _ { i } },
\end{align}
and
$$\nu _{\Psi}  = \sum _ { \eta \in \mathscr { P } _ { N } } \nu  _{\Psi,\eta},   \quad \text {where}  \quad   \nu  _{\Psi,\eta}  = \sum _ { i = 1 } ^{N} | s _ { i } ^{\eta} |  \nu _ { \Psi,\eta _ { i } } .$$

 The following theorem  is our second main result: Characterizing compact linear combination of composition operators via Carleson measure.
\begin{thm}
Let $\Psi$ be an Orlicz function satisfying the $\Delta_{2}$-condition, $\varphi_i\in S(\mathbb{D})$ and $a_i \in \mathbb{C} \backslash \{0\}$ for $i\in\Lambda _ { N }$. Then the following statements are equivalent:

{\rm (a)} $T$ is compact on $A^{\Psi}(\mathbb{D})$;

{\rm (b)} $\mu_{\Psi}$ is a vanishing Carleson measure on $\mathbb{D}$;

{\rm (c)} $\nu_{\Psi}$ is a vanishing Carleson measure on $\mathbb{D}$.
\end{thm}

This paper is organized as follows. In Section 2, we introduce some basic facts and  estimates on Orlicz functions that will be used throughout the paper. Section 3 is devoted to the proof of Theorem 1.1. In Section 4, we give the proof of Theorem 1.2.

Throughout the paper, we use the letter $C$ to denote various positive constants which may vary at each occurrence.
The notation $A\lesssim B$ ($B\gtrsim A$) means that there exists an absolute positive constant $C$  such that $A\leq CB$. We write $A\approx B$ means that both $A\lesssim B$ and $B\lesssim A$ hold.

\section{Preliminaries}
In this section, we list some basic facts on the pseudohyperbolic distance and establish some lemmas that will be used in later sections.

\subsection{Pseudohyperbolic distance}

Let $\rho(a,z)$ be the pseudohyperbolic distance between $a$ and $z$, and
$$\rho(a,z)=|\varphi_{a}(z)|=\left|\frac{a-z}{1-a\bar{z}}\right|, \ a, z\in\mathbb{D}.$$
By a simple calculation, we have
\begin{align}\label{eq: distance}
1-\rho^{2}(a,z)=\frac{(1-|a|^{2})(1-|z|^{2})}{|1-a\bar{z}|^{2}}.
\end{align}
For  $a\in\mathbb{D}$ and $0<r<1$, $E_{r}(a) $  denotes the pseudohyperbolic disk centered at $a$ with radius $r$. That is, $$E_{r}(a) =\{z\in\mathbb{D}, \rho(a,z)<r\}.$$
In fact, $E_{r}(a)$ is a Euclidean disk centered at $\frac{1-r^{2}}{1-|a|^{2}r^{2}}a$ with radius $\frac{(1-|a|^{2})}{1-|a|^{2}r^{2}}r$.
For all $a,z\in\mathbb{D}$, by \eqref{eq: distance}, we have
\begin{align}\label{eq: ratio-estimate}
\frac{1-\rho(a,z)}{1+\rho(a,z)}\leq\frac{1-|a|}{1-|z|}\leq\frac{1+\rho(a,z)}{1-\rho(a,z)}.
\end{align}
Given $0<r<1$, it follows that
\begin{align}\label{eq: area-estimate}
  A(E_{r}(a)) & \approx(1-|a|^{2})^{2},      \quad  a\in \mathbb{D}.
\end{align}
For $a\in\mathbb{D}$, $0<r<1$ and $z\in E_{r}(a)$, we have
\begin{align}\label{eq: kernel-estimate}
 |1-\bar{a}z| \approx 1-|a|^{2}\approx  1-|z|^{2}.
\end{align}
The constant in formula \eqref{eq: area-estimate}  and \eqref{eq: kernel-estimate} depend only on $r$.
All the details for the statements above can be found in \cite[Sections 4.1 and 4.2]{Zh}.

\subsection{Some Lemmas}

In this subsection, we present some results on Orlicz functions in Bergman-Orlicz space $A^\Psi(\mathbb{D})$.

The following lemma characterizes the compactness of a linear combination of composition operators on $A^\Psi(\mathbb{D})$.
The compactness of a single composition operator on Bergman-Orlicz spaces over the unit disk can be found in {\rm \cite[Proposition 5.4]{LLQR2}}. The proof can be easily modified for the operator $T$.
\begin{lem}\label{le: sequence-compact}
Let $T$ be a linear combination of composition operators as described in Section 1. Then $T$ is compact on $A^{\Psi}(\mathbb{D})$ if and only if for every bounded sequence $\{f_{n}\}$ in $A^{\Psi}(\mathbb{D})$ that converges to $0$ uniformly on compact subsets of $\mathbb{D}$, we have $\|Tf_{n}\|_{A^\Psi(\mathbb{D})}\rightarrow0$ as $n\rightarrow \infty$.
\end{lem}

The following lemma plays an important role in Bergman-Orlicz spaces. It translates norm convergence into an integral convergence, and serves as a foundational tool for studying the compactness of operators.

\begin{lem}\label{le: norm-convergence}
Let $\Psi$ be an Orlicz function satisfying the $\Delta_{2}$-condition. Then, for any sequence $\{f_{n}\}$ in $A^{\Psi}(\mathbb{D})$,  $\lim_{n\to\infty}\|f_n\|_{A^\Psi(\mathbb{D})}=0$ if and only if
$$\lim _ { n \to \infty } \int _ { \mathbb{D} }\Psi ( | f _ { n } (z) | )    {\rm d}A(z)  = 0 .$$
\end{lem}
\begin{proof}
By the Remark following {\rm Proposition 3.14} in {\rm \cite{LLQR2}}, for any sequence $\{f_{n}\}$ in $A^{\Psi}(\mathbb{D})$, $\lim\limits_{n\to\infty}\|f_n\|_{A^\Psi(\mathbb{D})}=0$ if and only if for every $\varepsilon>0$,
$$\lim _ { n \to \infty } \int _ { \mathbb{D} }\Psi \bigg( \frac{| f _ { n } |}{\varepsilon} \bigg)    {\rm d}A  = 0 .$$
Moreover, applying \eqref{eq: remark}, there exists a constant $C_{\varepsilon}>0$ such that
$$ \frac{1}{C_{\varepsilon}}  \Psi ( | f _ { n } | ) \leq\Psi\bigg( \frac{| f _ { n } |}{\varepsilon} \bigg)\leq C_{\varepsilon} \Psi ( | f _ { n } | ).$$
Hence, if $\Psi$ satisfies the $\Delta_{2}$-condition, then for any sequence $\{f_{n}\}$ in $A^{\Psi}(\mathbb{D})$, $\lim\limits_{n\to\infty}\|f_n\|_{A^\Psi(\mathbb{D})}=0$ if and only if
$$\lim _ { n \to \infty } \int _ { \mathbb{D} }\Psi ( | f _ { n } | )    {\rm d}A  = 0 .$$
\end{proof}

In what follows, we construct a test function in the Bergman-Orlicz space $A^{\Psi}(\mathbb{D})$ and estimate its Luxemburg norm.
\begin{lem}\label{le: test-function}
Let $\Psi$ be an Orlicz function satisfying the positive upper type $p$ condition for some $0<p<\infty$. For $z\in\mathbb{D}$ and  $s>2$, define $$u_{z}(\omega)=\bigg(\frac{1-|z|^{2}}{1-\bar{z}\omega}\bigg)^{s},       \,\,\,\,\,\,   \omega\in\mathbb{D}.$$ Then $u_{z}\in A^{\Psi}(\mathbb{D}) $ and
$$\|u_{z}\|_{ A^{\Psi}(\mathbb{D}) } \approx\frac{1}{\Psi^{-1}\big(\frac{1}{(1-|z|^{2})^{2}}\big)}.$$
\end{lem}
\begin{proof}
Fix $0<r<1$. For every Orlicz function  $\Psi$, we have
$$ 1\geq\int _ { \mathbb{D}}\Psi\left(\frac { |  u_{z}(\omega) | } { \| u_{z} \| _ { A ^{\Psi} (\mathbb{D} ) } } \right)  {\rm d}A(\omega)\geq  \int _ { E_{r}(z)}\Psi\left(\frac { | u_{z}(\omega) | } { \| u_{z}\| _ { A ^{\Psi} (\mathbb{D} ) } } \right)  {\rm d}A(\omega).$$
For $\omega\in E_{r}(z)$, it follows from \eqref{eq: kernel-estimate} that
 $$| u_{z}(\omega) |=\bigg(\frac{1-|z|^{2}}{|1-\bar{z}\omega|}\bigg)^{s}\approx\bigg(\frac{1-|z|^{2}}{1-|z|^{2}}\bigg)^{s}=1.$$
Applying \eqref{eq: area-estimate}, we obtain
\begin{align*}
  \int _ { E_{r}(z)}\Psi\left(\frac { | u_{z}(\omega) | } { \| u_{z}\| _ { A ^{\Psi} (\mathbb{D} ) } } \right)  {\rm d}A(\omega)
  &\approx\int _ { E_{r}(z)}\Psi\left(\frac { 1 } { \| u_{z}\| _ { A ^{\Psi} (\mathbb{D} ) } } \right)  {\rm d}A(\omega)\\
  &\approx\Psi\left(\frac { 1 } { \| u_{z}\| _ { A ^{\Psi} (\mathbb{D} ) } } \right)(1-|z|^{2})^{2}.
\end{align*}
Hence
$$\Psi\left(\frac { 1 } { \| u_{z}\| _ { A ^{\Psi} (\mathbb{D} ) } } \right)\lesssim \frac{1}{(1-|z|^{2})^{2}}.$$
That is,  $$\| u_{z}\| _ { A ^{\Psi} (\mathbb{D} ) }\gtrsim\frac{1}{ \Psi^{-1}\big(\frac{1}{(1-|z|^{2})^{2}}\big)}.$$

On the other hand, since $|1-\bar{z}\omega| \geq 1-|z||\omega| \geq1-|z|$, we have
$\frac{1-|z|^{2}}{|1-\bar{z}\omega|}\leq\frac{1-|z|^{2}}{1-|z|}=1+|z|\leq2$. Let $$\lambda=\frac{1}{ \Psi^{-1}\big(\frac{1}{(1-|z|^{2})^{2}}\big)}.$$
By the convexity of $\Psi$ and the positive upper type $p$ condition, we have
\begin{align*}
\int _ { \mathbb{D}}\Psi\left(\frac { |  u_{z}(\omega) | } {\lambda } \right) {\rm d}A(\omega)
   &=\int _ { \mathbb{D}}\Psi\bigg[\Big(\frac{1-|z|^{2}}{2|1-\bar{z}\omega|} \Big)^{s} 2^{s}\Psi^{-1}\Big(\frac{1}{(1-|z|^{2})^{2}}\Big)\bigg]{\rm d}A(\omega)\\
   &\leq\int _ { \mathbb{D}} \Big(\frac{1-|z|^{2}}{2|1-\bar{z}\omega|} \Big)^{s} \Psi\bigg[ 2^{s}\Psi^{-1}\Big(\frac{1}{(1-|z|^{2})^{2}}\Big)\bigg]{\rm d}A(\omega)\\
   &\lesssim \int _ { \mathbb{D}} \Big(\frac{1-|z|^{2}}{2|1-\bar{z}\omega|} \Big)^{s}2^{ps} \Psi\bigg[ \Psi^{-1}\Big(\frac{1}{(1-|z|^{2})^{2}}\Big)\bigg]{\rm d}A(\omega)\\
   &\approx (1-|z|^{2})^{s-2} \int _ { \mathbb{D}} \frac{1}{|1-\bar{z}\omega|^{s}} {\rm d}A(\omega)\\
   &\thickapprox (1-|z|^{2})^{s-2} (1-|z|^{2})^{-s+2}\\
   &= 1.
\end{align*}
This yields $$\| u_{z}\| _ { A ^{\Psi} (\mathbb{D} ) }\leq \lambda=\frac{1}{ \Psi^{-1}\big(\frac{1}{(1-|z|^{2})^{2}}\big)}.$$
Therefore
$$\| u_{z}\| _ { A ^{\Psi} (\mathbb{D} ) }\thickapprox\frac{1}{ \Psi^{-1}\big(\frac{1}{(1-|z|^{2})^{2}}\big)}.$$
\end{proof}

We next study several properties of functions in $A^{\Psi}(\mathbb{D})$, which can be viewed as extensions of the corresponding properties of functions in Bergman spaces.

\begin{lem}\label{le: properties-of-functions}
Let $\Psi$ be an Orlicz function and $f \in H(\mathbb{D})$. Then the following assertions hold:

{\rm (i)} If $0<r<1$, then there exists a constant $C=C_{r}>0$ such that
$$\Psi(|f(a)|)\leq\frac{ C }{A( E_{r}(a) )} \int _ { E_{r}(a) } \Psi(|f|)  {\rm d} A.$$

{\rm (ii)} If $0 <r_{1}<r_{2}< 1$, $\Psi$ satisfies the $\Delta_{2}$-condition. Then there exists a constant $C=C(r_{1},r_{2})>0$ such that $$\Psi(| f (a) - f ( z ) |) \leq C\frac { \rho ( a , z ) } {( 1-|a|^{2} )^{2}} \int _ { E _ { r_{2}} (a) } \Psi(| f | ){\rm d} A,     \,\,\,\,\,\,\, z\in E _ { r_{1}} (a).$$
\end{lem}
\begin{proof}
{\rm (i)}
For any $ f \in H(\mathbb{D})$, the function $|f\circ\varphi_{a}|$ is subharmonic. Hence
$$|f\circ\varphi_{a}(0) | \leq\frac { 1 } { A( E_{r}(0) ) } \int _ { E_{r}(0) } | f\circ\varphi_{a}(z) |  {\rm d} A(z) .$$
By the non-decreasing property of $\Psi$ and Jensen's inequality, we have
\begin{align*}
\Psi(|f\circ\varphi_{a}(0)  |) &\leq \Psi\bigg(\frac { 1 } { A( E_{r}(0) ) } \int _ { E_{r}(0) } | f\circ\varphi_{a}(z) |  {\rm d} A(z)\bigg)\\
                      &\leq \frac { 1 } { A( E_{r}(0) ) } \int _ { E_{r}(0) } \Psi(| f\circ\varphi_{a}(z) | ) {\rm d} A(z).
\end{align*}
Let $\omega=\varphi_{a}(z)$. Applying \eqref{eq: area-estimate} and \eqref{eq: kernel-estimate}, we obtain
\begin{align*}
\Psi(|f(a)|) &\leq \frac { 1 } { A( E_{r}(0) ) } \int _ { E_{r}(a) } \Psi(|f(\omega)|) |\varphi_{a}^{'}(\omega)|^{2} {\rm d} A(\omega)\\
           &= \frac { 1 } { A( E_{r}(0) ) } \int _ { E_{r}(a) } \Psi(|f(\omega)|) \frac{(1-|a|^{2})^{2}}{|1-\bar{a}\omega|^{4}} {\rm d} A(\omega)\\
          &\thickapprox\frac{ 1 }{A( E_{r}(a) )} \int _ { E_{r}(a) } \Psi(|f(\omega)|)  {\rm d} A(\omega).
\end{align*}
Thus
\begin{align*}
\Psi(|f(a)|)&\leq\frac{ C }{A( E_{r}(a) )} \int _ { E_{r}(a) } \Psi(|f(\omega)|)  {\rm d} A(\omega),
\end{align*}
where $C$ is a constant depending only on $r$.

{\rm (ii)}
Fix $z\in E _ { r_{1}} (a)$. Since $0\leq\rho< 1$, by the non-decreasing property and convexity of $\Psi$, Jensen's inequality, and the $\Delta_{2}$-condition, we obtain
\begin{align*}
\Psi(| f (a) - f (z) |)  & \leq\Psi\bigg( C(r_{1},r_{2})  \frac { \rho ( a , z ) } {( 1-|a|^{2} )^{2}} \int _ { E _ { r_{2}} (a) } | f | {\rm d} A  \bigg)\\
                          & \lesssim C(r_{1},r_{2})\frac { \rho ( a , z ) } {( 1-|a|^{2} )^{2}} \int _ { E _ { r_{2}} (a) }\Psi( | f | ){\rm d} A .
\end{align*}
So
$$\Psi(| f (a) - f ( z ) |) \leq C\frac { \rho ( a , z ) } {( 1-|a|^{2} )^{2}} \int _ { E _ { r_{2}} (a) } \Psi(| f| ){\rm d} A .$$
\end{proof}

Let $\varphi\in S(\mathbb{D})$, $f\in H(\mathbb{D})$, the connection between composition operators and Carleson measures comes from the standard identity
\begin{align}\label{eq: variable-change formula}
\int _ { \mathbb{D} } \Psi( | f \circ \varphi (z) |) {\rm d} A  (z) = \int _ { \mathbb{D} } \Psi(| f (z) | )  {\rm d}( A  \circ \varphi ^{- 1})  (z),
\end{align}
where $ A  \circ \varphi ^{- 1}$ denotes the pullback measure defined by $ (A  \circ \varphi ^{- 1})(E)=A(\varphi ^{- 1}(E))$ for Borel sets $E\subset\mathbb{D} $.

The following two lemmas play an important role in the proof of Theorem 1.1 and Theorem 1.2.

\begin{lem}\label{le: Orlicz-Carleson}
Let $\Psi$ be an Orlicz function satisfying the $\Delta_{2}$-condition, $\varphi \in S(\mathbb{D})$ and let $H: \mathbb{D} \to [0, 1]$ be a Borel function. For any $\varepsilon> 0$ and set $\gamma=\frac{1}{2}$, if
$$\sup_ { z \in \mathbb { D } } \left[ H (z) R _ { \varphi } (z) \right] \leq\varepsilon ,$$
then there exists a constant $C>0$ such that
$$\int _ { \mathbb { D } } \Psi(| f \circ \varphi |)  H  {\rm d} A  \leq C \varepsilon^{\gamma} \int _ { \mathbb { D } } \Psi(| f |)  {\rm d} A  $$
 for any $f\in A^{\Psi}(\mathbb{D})$.
\end{lem}
\begin{proof}
For any Borel set  $E\subset\mathbb { D } $, let $\mu(E)=\int_{\varphi^{-1}(E)} H {\rm d} A$. By \eqref{eq: variable-change formula}, we have$$\int _ { \mathbb { D }} \Psi(| f \circ \varphi |)  H  {\rm d} A  =\int _ { \mathbb { D }} \Psi(|f|)  H  {\rm d} A \circ \varphi ^{-1}=\int _ { \mathbb { D }} \Psi(|f|) {\rm d} \mu.$$
Let $W=W(\xi,h)$, taking 1 to be an upper bound of $H$, then we have
\begin{align*}
  \mu(W) & = \int_{\varphi^{-1}(W)} H(z) {\rm d} A(z)\\
   & \leq \varepsilon^{\gamma}\int_{\varphi^{-1}(W)} \bigg(\frac{1-|\varphi(z)|^{2}}{1-|z|^{2}}\bigg)^{\gamma} {\rm d} A(z).
\end{align*}
Note that $z\in \varphi^{-1}(W)$, so $1-h<|\varphi(z)|<1$, i.e.,  $1-|\varphi(z)|<h$, and $1-|\varphi(z)|^{2}<2h$. Therefore
\begin{align*}
  \mu(W) & \leq  \varepsilon^{\gamma} (2h)^{\gamma} \int_{\varphi^{-1}(W)} (1-|z|^{2})^{-\gamma} {\rm d} A(z)\\
   & =c_{\gamma}\varepsilon^{\gamma} (2h)^{\gamma} \int_{\varphi^{-1}(W)}  {\rm d} A_{-\gamma}(z)\\
   & = c_{\gamma}\varepsilon^{\gamma} (2h)^{\gamma} (A_{-\gamma}\circ\varphi^{-1}) (W).
\end{align*}
Hence
$$\frac{\mu(W)}{h^{2}}\leq c_{\gamma}\varepsilon^{\gamma} (2h)^{\gamma} \frac{(A_{-\gamma}\circ\varphi^{-1}) (W)}{h^{2}}.$$

Since $C_{\varphi}$ is bounded on $A_{-\gamma}^{\Psi}(\mathbb { D })$ and $\Psi$ satisfies the $\Delta_{2}$-condition, it follows from {\rm\cite[Theorem 3.3]{Ch1}} that $A_{-\gamma}\circ\varphi^{-1}$  is a Carleson measure.
Consequently, there exists a constant $C_{1}>0$ such that $$\int _ { \mathbb { D }} \Psi(|f|) {\rm d} \mu\leq C_{1} c_{\gamma}(2h)^{\gamma} \varepsilon^{\gamma} \int _ { \mathbb { D } } \Psi(| f |)  {\rm d} A.$$
Thus
 $$\int _ { \mathbb { D } } \Psi(| f \circ \varphi |)  H  {\rm d} A  \leq C \varepsilon^{\gamma} \int _ { \mathbb { D } } \Psi(| f |)  {\rm d} A  ,$$
where the constant $C$ depends only on $\gamma$.
\end{proof}

\begin{lem}\label{le: symbols-difference}
Let $\Psi$ be an Orlicz function satisfying the $\Delta_{2}$-condition, $\varphi, \psi \in S(\mathbb{D})$ and $K \subset \mathbb{D}$ is a Borel set. If
$$\sup _ { z \in K } M _ { \varphi , \psi } (z) \leq \varepsilon  \,\,\, for\,\, some \,\,\, \varepsilon > 0, $$
then there exists a constant $h(\varepsilon) > 0$, such that $\lim\limits _ { \varepsilon\to 0 } h ( \varepsilon ) = 0$   and
$$\int _ { K }\Psi(| f \circ \varphi - f \circ \psi | ) {\rm d} A  \leq h ( \varepsilon) \int _ { \mathbb{D}}\Psi(| f |) {\rm d} A $$
for any $f\in A^{\Psi}(\mathbb{D})$.
\end{lem}
\begin{proof}
Assume that $\sup\limits _ { z \in K } M _ { \varphi , \psi } (z) =\sup\limits  _ { z \in K }(R_{\varphi}(z)+R_{\psi}(z))\rho_{\varphi,\psi}(z)\leq \varepsilon$. For $\delta = \delta(\varepsilon )\in (0, \frac{1}{2})$ be chosen later, put
$$\Omega _ { \delta } = \left\{ z \in K : \rho_{\varphi,\psi}(z) < \delta \right\},  \,\,\,\, \Omega _ { \delta } ^{\prime}  = K \setminus \Omega _ { \delta } .$$
For any $ f \in A^{\Psi}(\mathbb{D})$, we write
\begin{align*}
  \int _ { K } \Psi(| f \circ \varphi(z) - f \circ \psi(z) | ) {\rm d} A(z)  & =I_{1}(f)+I_{2}(f),
\end{align*}
where $$I_{1}(f)=\int _ { \Omega _ { \delta } }\Psi(| f \circ \varphi(z) - f \circ \psi(z) | ) {\rm d} A(z) ,$$
 $$I_{2}(f)=\int _ { \Omega _ { \delta } ^{\prime} }\Psi(| f \circ \varphi(z) - f \circ \psi (z)| ) {\rm d} A(z). $$

We first estimate $I_{2}(f)$. For $z\in\Omega _ { \delta } ^{\prime}$, we have $\delta\chi_{\Omega_{\delta}^{\prime}}(z)\leq\rho_{\varphi,\psi}(z)$.
Hence
$$\chi _ { \Omega _ { \delta } ^{\prime} } (z) ( R_{\varphi}(z)+R_{\psi}(z)) \leq \frac{1}{\delta}( R_{\varphi}(z)+R_{\psi}(z))\rho_{\varphi,\psi}(z)\leq \frac { \varepsilon} { \delta }.$$
Let $$H(z)=\chi _ { \Omega _ { \delta } ^{\prime} } (z) .$$
So $$\sup_{z\in\mathbb{D}}H(z)R_{\varphi}(z)\leq\sup_{z\in\mathbb{D}}H(z)( R_{\varphi}(z)+R_{\psi}(z)) \leq \frac { \varepsilon} { \delta }.$$
Similarly, we can obtain $$\sup_{z\in\mathbb{D}}H(z)R_{\psi}(z) \leq \frac { \varepsilon} { \delta }.$$
Thus, by the convexity of $\Psi$, the $\Delta_{2}$-condition, and  Lemma \ref{le: Orlicz-Carleson}, we have
\begin{align}\label{eq: II-f}
  I_{2}(f) &=\int _ { \mathbb{D} }\Psi(| f \circ \varphi(z) - f \circ \psi(z) | )\chi _ { \Omega _ { \delta } ^{\prime} } (z){\rm d} A (z)  \nonumber\\
   &\leq \int _ { \mathbb{D} }\Psi(| f \circ \varphi(z) | +| f \circ \psi(z) | ) H (z){\rm d} A (z)  \nonumber\\
   & \lesssim \int _ {  \mathbb{D}}\Psi(| f \circ \varphi (z)| )H (z){\rm d} A(z) +\int _ {  \mathbb{D}} \Psi(| f \circ \psi(z) | )H(z) {\rm d} A(z)  \nonumber\\
   & \lesssim\left( \frac {\varepsilon } { \delta } \right) ^{\gamma} \int _ { \mathbb { D } } \Psi(| f(\omega) |)  {\rm d} A(\omega) .
\end{align}

Next, We estimate $I_{1}(f)$.
For $z\in\Omega _ { \delta }$, applying  Lemma \ref{le: properties-of-functions}{(ii)}  (put $r_1 = 1/2$ and $r_2 = 2/3$), we get
\begin{align*}
 I_{1}(f) & \lesssim \int _ { \Omega _ { \delta } }\frac { \rho_{\varphi,\psi}(z) } {( 1-|\varphi(z)|^{2} )^{2}} \int _ { E _ { 2 / 3} ( \varphi(z) ) } \Psi(| f(\omega) | ){\rm d} A (\omega) {\rm d} A (z)\\
   & < \delta\int _ { \mathbb{D} }\frac { \chi_{ \Omega _ { \delta }}(z) } {( 1-|\varphi(z)|^{2} )^{2}} \int _ {\mathbb{D}}\chi_ { E _ { 2 / 3} ( \varphi(z) ) } (\omega)\Psi(| f(\omega) | ){\rm d} A (\omega) {\rm d} A (z).
\end{align*}
Furthermore, by \eqref{eq: area-estimate}, \eqref{eq: kernel-estimate} and {\rm Fubini's }theorem, we deduce that
\begin{align}\label{eq: I-f}
   I_{1}(f) & \lesssim \delta\int _ { \mathbb{D} }\frac { \chi_{ \Omega _ { \delta }}(z)  } {( 1-|\omega|^{2} )^{2}} \int _ {\mathbb{D}}\chi_ { \varphi^{-1}( E _ { 2 / 3} (\omega))} (z)\Psi(| f(\omega) | ){\rm d} A (\omega) {\rm d} A (z)   \nonumber\\
   & =\delta\int _ { \mathbb{D} }\frac {\Psi(| f(\omega) | )} {( 1-|\omega|^{2} )^{2}} \int _ { \Omega _ { \delta }\cap\varphi^{-1}( E _ { 2 / 3} (\omega))} {\rm d} A (z){\rm d} A (\omega)     \nonumber\\
   &\lesssim\delta\int _ { \mathbb{D} }\Psi(| f(\omega)| )  \frac{(A\circ\varphi^{-1})( E _ { 2 / 3} (\omega))}{A( E _ { 2 / 3} (\omega))}d A (\omega)       \nonumber\\
   &\leq\delta\sup_{\omega\in\mathbb{D}} \frac{(A\circ\varphi^{-1})( E _ { 2 / 3} (\omega))}{A( E _ { 2 / 3} (\omega))}\int _ { \mathbb{D} }\Psi(| f(\omega) | ) {\rm d} A (\omega)       \nonumber\\
   &\lesssim\delta\int _ { \mathbb{D} }\Psi(| f(\omega) | ) {\rm d} A (\omega).
\end{align}
The last inequality holds due to $C_{\varphi}$ is bounded on $A^{\Psi}(\mathbb{D})$ and $\Psi$ satisfies the $\Delta_{2}$-condition, by Theorem $\mathrm {A}$, we have $A\circ\varphi^{-1}$ is a Carleson measure.
So
$$\sup_{\omega\in\mathbb{D}} \frac{(A\circ\varphi^{-1})( E _ { 2 / 3} (\omega))}{A( E _ { 2 / 3} (\omega))}<\infty.$$

 \eqref{eq: II-f} together with \eqref{eq: I-f} yields
\begin{align*}
 \int _ { K } \Psi(| f \circ \varphi(z) - f \circ \psi(z) | ) {\rm d} A(z)  &\leq C [( \frac {\varepsilon } { \delta } ) ^{\gamma}+\delta ]\int _ { \mathbb { D } } \Psi(| f(\omega) |)  {\rm d} A(\omega) .
\end{align*}
Taking $\delta=\frac{\sqrt{\varepsilon}}{1+2\sqrt{\varepsilon}}$, $h(\varepsilon)=C [( \frac {\varepsilon } { \delta } ) ^{\gamma}+\delta ]$, we have  $h(\varepsilon)\rightarrow0$ as $\varepsilon\rightarrow0$, and
$$\int _ { K } \Psi(| f \circ \varphi(z) - f \circ \psi(z) | ) {\rm d} A(z) \leq h(\varepsilon)\int _ { \mathbb { D } } \Psi(| f(\omega) |)  {\rm d} A(\omega). $$
\end{proof}

The following lemma plays a vital role in the proof of Theorem 1.2.
\begin{lem}\label{le: vanishing-Carleson-measure}
Let $\Psi$ be an Orlicz function satisfying the $\Delta_{2}$-condition, $\varphi, \psi \in S(\mathbb{D})$ and $H : \mathbb{D} \to [0,1]$ be a Borel function. For every Borel set $E\subset \mathbb{D}$, set
$$\mu (E) =  \int _ { \varphi ^{- 1} (E) }\Psi ( M _ { \varphi, \psi  } )  H {\rm d} A +  \int _ { \psi ^{- 1}(E) } \Psi ( M _ { \varphi, \psi  } )  H  {\rm d} A $$
and
$$\nu (E) = \int _ { \varphi ^{- 1} (E)} \Psi ( \rho _ { \varphi, \psi  } ) H  {\rm d} A  + \int _ { \psi ^{- 1} (E)} \Psi ( \rho _ { \varphi, \psi  } ) H  {\rm d} A .$$
 If $\mu$ is a vanishing Carleson measure on $\mathbb{D}$, then so is $\nu$.
\end{lem}
\begin{proof}
For $0<\varepsilon<1$, let $K_{\varepsilon}=\{z\in\mathbb{D}: R_{\varphi}(z)+R_{\psi}(z)\leq\varepsilon\}$. For any Borel set $E\subset \mathbb{D}$, put
\begin{align*}
  \int _ { \varphi ^{- 1} (E)} \Psi ( \rho _ { \varphi, \psi  } ) H {\rm d} A  & =\int _ { \varphi ^{- 1} (E)\cap K_{\varepsilon}} \Psi ( \rho _ { \varphi, \psi  } ) H  {\rm d} A +\int _ { \varphi ^{- 1} (E)\setminus K_{\varepsilon}} \Psi ( \rho _ { \varphi, \psi  } ) H  {\rm d} A  \\
   & =\int _ {\varphi^{- 1} (E) } \chi_{ K_{\varepsilon}}\Psi ( \rho _ { \varphi, \psi  } ) H  {\rm d} A+\int _ { \varphi ^{- 1} (E)\setminus K_{\varepsilon}} \Psi ( \rho _ { \varphi, \psi  } ) H  {\rm d} A .
\end{align*}

For $z\notin K_{\varepsilon}$, we have $R_{\varphi}(z)+R_{\psi}(z)>\varepsilon$, so $ \rho _ { \varphi,\psi }(z) <\frac{M _ { \varphi,\psi }(z)}{\varepsilon}$ . Since $\Psi$ is non-decreasing, it follows that $\Psi ( \rho _ { \varphi,\psi } ) <\Psi ( \frac{M _ { \varphi,\psi }}{\varepsilon})$. By the $\Delta_{2}$-condition, there exists a constant $C_{\varepsilon}>0$ such that
\begin{align*}
 \int _ { \varphi ^{- 1} (E)\setminus K_{\varepsilon}} \Psi ( \rho _ { \varphi, \psi  } ) H  {\rm d} A & <\int _ { \varphi ^{- 1} (E)\setminus K_{\varepsilon}} \Psi \left( \frac{M _ { \varphi,\psi }}{\varepsilon}\right) H  {\rm d} A\\
   & \leq C_{\varepsilon}\int _ { \varphi ^{- 1} (E)}\Psi(M_{\varphi,\psi}) H  {\rm d} A.
\end{align*}
Hence
\begin{align*}
  \int _ { \varphi ^{- 1} (E)}\Psi ( \rho _ { \varphi, \psi  } ) H  {\rm d} A  & \leq\int _ {\varphi^{- 1} (E) } \chi_{ K_{\varepsilon}}\Psi ( \rho _ { \varphi, \psi  } )H  {\rm d} A  +C_{\varepsilon}\int _ { \varphi ^{- 1} (E)}\Psi(M_{\varphi,\psi})H  {\rm d} A.
\end{align*}
By interchanging the roles of $\varphi$ and $\psi$, we obtain
\begin{align*}
  \int _ { \psi ^{- 1} (E)} \Psi ( \rho _ { \varphi, \psi  } ) H  {\rm d} A  & \leq\int _ {\psi^{- 1} (E) } \chi_{ K_{\varepsilon}}\Psi ( \rho _ { \varphi, \psi  } ) H  {\rm d} A  +C_{\varepsilon}\int _ { \psi ^{- 1} (E)}\Psi(M_{\varphi,\psi}) H  {\rm d} A.
\end{align*}

Let
$$\nu_{\varepsilon}(E) =\int _ {\varphi^{- 1} (E) } \chi_{ K_{\varepsilon}}\Psi ( \rho _ { \varphi, \psi  } ) H  {\rm d} A +\int _ {\psi^{- 1} (E) } \chi_{ K_{\varepsilon}}\Psi ( \rho _ { \varphi, \psi  } ) H  {\rm d} A .$$
Therefore
$$\nu(E) \leq\nu_{\varepsilon}(E) +C_{\varepsilon}\mu (E) .$$
Taking $E=W(\xi,h)$ in the above inequality, we obtain
$$\frac{\nu(W(\xi,h))}{h^{2}} \leq\frac{\nu_{\varepsilon}(W(\xi,h))}{h^{2}}+C_{\varepsilon}\frac{\mu (W(\xi,h))}{h^{2}}.$$
Since $\mu$ is a vanishing Carleson measure on $\mathbb{D}$, we get
\begin{align*}
\lim_{h\rightarrow0}\sup_{\xi\in\mathbb{T}}\frac{\nu(W(\xi,h))}{h^{2}} &\leq\lim_{h\rightarrow0}\sup_{\xi\in\mathbb{T}}\frac{\nu_{\varepsilon}(W(\xi,h))}{h^{2}}.
\end{align*}

It remains to show that $\lim_{h\rightarrow0}\sup_{\xi\in\mathbb{T}}\frac{\nu_{\varepsilon}(W(\xi,h))}{h^{2}}=0$,
that is, $\nu_{\varepsilon}$ is a vanishing Carleson measure. Since $\nu_{\varepsilon}$ is a vanishing Carleson measure if and only if the embedding $I_{\nu_{\varepsilon}} :A^{\Psi}(\mathbb{D})\rightarrow L^{\Psi}({\rm d}\nu_{\varepsilon})$ is compact. To prove the compactness of $I_{\nu_{\varepsilon}} $, it suffices to show that for any $ \{f_{n}\}$ in $A^{\Psi}(\mathbb{D})$ such that $\|f_{n}\|_{A^{\Psi}(\mathbb{D})}\leq1$ and $f_{n}\to0$ uniformly on compact sets of $\mathbb{D}$, we have $\|I_{\nu_{\varepsilon}} f_{n}\|_{L^{\Psi}({\rm d}\nu_{\varepsilon})}=\| f_{n}\|_{L^{\Psi}({\rm d}\nu_{\varepsilon})}\rightarrow0 \,\,\,\,as \,\,\,n\rightarrow\infty$.
By Lemma \ref{le: norm-convergence}, it remains to prove that $\int _ {\mathbb{D}}\Psi(|f_{n}|) {\rm d}\nu_{\varepsilon}\rightarrow0 \,\,\, (n\rightarrow\infty)$.

Since $0\leq \rho<1$ and $\Psi$ is non-decreasing, we obtain
\begin{align*}
  \int _ {\mathbb{D}}\Psi(|f_{n}|) {\rm d}\nu_{\varepsilon} & =\int _ {\mathbb{D}}\Psi(|f_{n}|)\bigg(\int _ {\varphi^{- 1} (E) } \chi_{ K_{\varepsilon}}\Psi ( \rho _ { \varphi,\psi } )H  {\rm d}A +\int _ {\psi^{- 1} (E) } \chi_{ K_{\varepsilon}}\Psi ( \rho _ { \varphi,\psi } )H  {\rm d} A \bigg)\\
   & \leq\Psi(1)\bigg[\int _ {\mathbb{D}}\Psi(|f_{n}|) \bigg( \chi_{ K_{\varepsilon}}H  {\rm d} (A\circ \varphi^{- 1})+\chi_{ K_{\varepsilon}}H  {\rm d}(A \circ\psi^{- 1}) \bigg)\bigg] \\
   & =\Psi(1)\bigg[\int _ {\mathbb{D}}\Psi(|f_{n}|) \chi_{ K_{\varepsilon}}H  {\rm d} (A\circ \varphi^{- 1})+\int _ {\mathbb{D}}\Psi(|f_{n}|)\chi_{ K_{\varepsilon}}H  {\rm d} (A \circ\psi^{- 1})\bigg]  \\
   & =\Psi(1)\bigg[\int _ {\mathbb{D}}\Psi(|f_{n}\circ \varphi|) \chi_{ K_{\varepsilon}}H  {\rm d} A+\int _ {\mathbb{D}}\Psi(|f_{n}\circ\psi|)\chi_{ K_{\varepsilon}}H  {\rm d} A\bigg] \\
   & =\Psi(1)(J_{1}+J_{2}).
\end{align*}
For $J_{1}$. Let $H_{1}=\chi_{ K_{\varepsilon}}H\leq1 $. So
$$\sup_{z\in\mathbb{D}}H_{1}(z)R_{\varphi}(z)\leq\sup_{z\in K_{\varepsilon}}R_{\varphi}(z)\leq\sup_{z\in K_{\varepsilon}}(R_{\varphi}(z)+R_{\psi}(z))\leq\varepsilon.$$
By Lemma \ref{le: Orlicz-Carleson}, there exists a constant $C_{1}>0$ such that
\begin{align*}
  J_{1} & =\int _ {\mathbb{D}}\Psi(|f_{n}\circ \varphi|) H_{1}  {\rm d} A \\
   & \leq C_{1}\varepsilon^{\gamma}\int _ {\mathbb{D}}\Psi(|f_{n}|) {\rm d} A \\
   & \leq C_{1}\varepsilon^{\gamma}.
\end{align*}
Meanwhile, we get $$J_{2}\leq C_{1}\varepsilon^{\gamma}.$$
Therefore $$\int _ {\mathbb{D}}\Psi(|f_{n}|) {\rm d}\nu_{\varepsilon}\leq 2\Psi(1)C_{1}\varepsilon^{\gamma}.$$
By the arbitrariness of $\varepsilon$, we have  $\int _ {\mathbb{D}}\Psi(|f_{n}|) {\rm d}\nu_{\varepsilon}\rightarrow 0$ as $n\to\infty$.
Hence, $\nu$ is a vanishing Carleson measure on $\mathbb{D}$.
\end{proof}

\section{ Julia-Carathéodory type characterization}
In this section, we prove Theorem 1.1. We first introduce some notations. Let
\begin{align}\label{eq: T-J}
T _ { J }  &= \sum _ { i \in J } a _ { i } T _ { i }
\end{align}
for $J \subset \Lambda_N$. In view of \eqref{eq: operator-T}, the operators  $T_J$ can be rewritten as
\begin{align}\label{eq: operators-T}
 T _ { J } = \sum _ { i , k \in J } c _ { i , k } T _ { i , k } + \left( \sum _ { i \in J } a _ { i } \right) T _ { \ell },
\end{align}
where $\ell\in J$.

Since $\Psi$ is an Orlicz function satisfying the $\Delta_{2}$-condition, there exists a constant $C=C_{N}>0$ such that
\begin{align}\label{eq: subadditivity}
  \Psi(\sum_{i=1}^{N} t_{i}) &\leq C \sum_{i=1}^{N}\Psi( t_{i})
\end{align}
for all $t_{i}>0,i\in\Lambda_N$.

\vskip.2cm

\textbf{Proof of Theorem 1.1}.
First, proof of ${\rm (a)}\Rightarrow {\rm (b)}$.
Assume that $T$ is compact on $A^{\Psi}(\mathbb{D}) $, let $$f_{z}=\frac {u_{z} } { \|u_{z}\|_{ A^{\Psi}(\mathbb{D}) }  },$$
 then $f_{z}\in A^{\Psi}(\mathbb{D}) $ and $\|f_{z}\|_{A^{\Psi}(\mathbb{D}) } =1 $, i.e., $\{f_{z}\}$ is bounded on $A^{\Psi}(\mathbb{D}) $.

Since $\frac { \Psi ( x ) } { x } \rightarrow\infty $ as $x \to \infty$. Let $t=\Psi(x)$, we get $\frac{\Psi^{-1}(t)}{t}\rightarrow0$ as $t \to \infty$. So
$$(1-|z| ^{2})^{2} \Psi^{-1}\big(\frac{1}{(1-|z|^{2})^{2}}\big) \rightarrow 0\ \ as \ \ |z| \rightarrow 1.$$
Meanwhile, $u_{z}$ is bounded on compact subsets of $\mathbb{D}$. Hence,
$$\frac {|u_{z}| } { \|u_{z}\|_{ A^{\Psi}(\mathbb{D}) }  }\lesssim(1-|z| ^{2})^{s} \Psi^{-1}\bigg(\frac{1}{(1-|z|^{2})^{2}}\bigg)$$  holds on compact subsets of  $\mathbb{D}$.
It follows from $s>2$ that $\{f_{z}\}$ converges to $0$ uniformly on compact subsets of  $\mathbb{D}$  as $|z| \rightarrow 1$. According to Lemma \ref{le: sequence-compact},  $\|Tf_{z}\|_{A^{\Psi}(\mathbb{D}) } \to 0$. i.e., $$\lim_{|z|\to1}\frac{\|Tu_{z}\|_{A^{\Psi}(\mathbb{D})}}{\|u_{z}\|_{ A^{\Psi}(\mathbb{D}) }}=0.$$

Next, proof of ${\rm (b)}  \Rightarrow {\rm (c)}$.  Assume that ${\rm (b)}$ holds. We suppose that ${\rm (c)}$ fails and will derive a contradiction.

Since ${\rm (c)}$ fails, there exists a sequence $\{z_n\} \subset \mathbb{D}$ such that $|z_n| \to 1$ and
$$\inf _ { n } Q ( z _ { n } ) > 0 .$$
Since the functions $Q_{\eta}$ are bounded on $\mathbb{D}$ and $Q ( z _ { n } ) =\prod_{\eta \in \mathscr{ P } _ { N }}Q_{\eta}(z_n)$, we obtain
\begin{align}\label{eq: inf-Q}
\inf _ { n } Q _ { \eta } ( z _ { n } ) > 0 ,   \,\,\,\,\,\,\,\,\eta \in \mathscr{ P } _ { N }.
\end{align}
Furthermore, we suppose
$\{ R _ { i } ( z _ { n } ) \} $ and $ \{ \rho _ { i , k } ( z _ { n } ) \} $  both converge as $n \to \infty$, for every $i,k\in\Lambda_{N}$. Let
\begin{align}\label{eq: define-F}
F = \left\{ i \in \Lambda _ { N } :\lim _ { n \to \infty } R _ { i } ( z _ { n } ) = 0 \right\}
\end{align}
and
$$F ^{\prime}  = \Lambda _ { N } \setminus F.$$
Assume that
\begin{align}\label{eq: inf-R}
\inf_ { n } R _ { i } ( z _ { n } ) > 0
\end{align}
for each $i \in F'$.

We interrupt the proof to show that the set $F'$ satisfies the following claim.

{\bf Claim}: There exist pairwise disjoint sets of indices $J_1, \ldots, J_L$ and a subsequence of $\{z_n\}$ satisfies the following properties:

{\rm (i)} $F'=\bigcup\limits_{m=1}^{L}J_{m}$;

{\rm (ii)} For each $m=1,\ldots,L $, there exists $k_m\in J_m $ such that for all $n $ and all $i \in \bigcup\limits _ { \ell = m } ^{L} J _ { \ell } $,
$$| \varphi _ { i } ( z _ { n } ) | \leq | \varphi _ { k _ { m } } ( z _ { n } ) |; $$

{\rm (iii)} $\lim\limits _{n \rightarrow \infty} \rho_{i, k}\left(z_{n}\right)=0$  only if $i, k \in J_{m}$ for  $m=1, \ldots, L$;

{\rm (iv)} $\sum\limits_{i \in J_{m}} a_{i}=0$ for all $m=1, \ldots, L$.

Proof of the Claim. The proofs of {\rm (i)}-{\rm (iii)} can be found in \cite{CKW}. It remains only to prove {\rm (iv)}. We first consider the case $m=1$. Assume that $\sum\limits_{i \in J_{1}} a_{i}\neq0$. Then we will reach a contradiction.  Let
$$f _ { a } = \Psi^{-1}\bigg( \frac{1}{(1 - | a | ^{2})^{2}}\bigg )u_{a}, \,\,\,\, a\in\mathbb{D}.$$
Set $$w _ { n } = \varphi _ { k _ { 1 } } ( z _ { n } )$$
for each $n$.  By \eqref{eq: inf-R}, we have $1 - |w_n|^2 \approx 1 - |z_n|^2$. Hence, $|w_n| \to 1$. Applying Lemma \ref{le: test-function}, we obtain
\begin{align*}
  \|  f _ { w _ { n } } \| _ { A ^{\Psi}(\mathbb{D}) }
   = \Psi^{-1}\bigg( \frac{1}{(1 - | w _ { n } | ^{2})^{2}}\bigg )\|  u_{w _ { n }} \| _ { A^{\Psi}(\mathbb{D}) }
    \approx1 .
\end{align*}
By Lemma \ref{le: test-function} and (b), we derive
\begin{align*}
\lim _ { n \to \infty } \| T f _ { w _ { n } } \| _ { A ^{\Psi}(\mathbb{D}) }
 &= \lim _ { n \to \infty } \| T u_{w _ { n }} \| _ { A ^{\Psi}(\mathbb{D}) }\Psi^{-1}\bigg( \frac{1}{(1 - | w _ { n } | ^{2})^{2}}\bigg )\\
  &\thickapprox \lim _ { n \to \infty } \frac{\| T u_{w _ { n }} \| _ { A ^{\Psi}(\mathbb{D}) }}{\|  u_{w _ { n }} \| _ { A ^{\Psi}(\mathbb{D}) }}=0.
\end{align*}
 Applying Lemma \ref{le: properties-of-functions} {\rm (i)}, we have
\begin{align*}
\Psi(| Tf _ { w _ { n } } ( z_{n} ) |) &\lesssim \frac { 1 } { ( 1 - | z_{n} |^{2}  ) ^{2} } \int _ { E_{r}(z_{n})  } \Psi(| Tf _ { w _ { n }} |) {\rm d}A \\
                                       &\lesssim\frac { 1 } { ( 1 - |w _ { n } |^{2}  ) ^{2} } \int _ { \mathbb{D}  } \Psi(| Tf _ { w _ { n }} |) {\rm d}A .
\end{align*}
Furthermore, Lemma \ref{le: norm-convergence} gives
\begin{align}\label{eq: lim-T}
\lim _ { n \to \infty }( 1 - | w _ { n }|^{2}  ) ^{2} \Psi(| Tf _ { w _ { n } } ( z_{n} ) |) = 0.
\end{align}

For $i,k_{1} \in J _ { 1 }$, by Claim {\rm (iii)}, we have $\lim \limits_{n \rightarrow \infty} \rho_{i, k_{1}}\left(z_{n}\right)=0$.
It follows from \eqref{eq: ratio-estimate} that
 $$1 - | w _ { n } | ^{2}= 1- |\varphi_{k_{1}}(z_{n}) |^{2} \approx1 - | \varphi _ { i } ( z _ { n } ) | ^{2} .$$
Let  $0<r_{1}<r_{2}<1$. By Lemma \ref{le: properties-of-functions} {\rm (ii)}, we have
\begin{align*}
\Psi(|T _ { i,k  }f _ { w _ { n } } ( z _ { n } ) |)
     & \lesssim \frac{\rho_{i,k} (z _ { n })}{(1 - | \varphi _ { i } ( z _ { n } ) | ^{2})^{2}} \int _ { E _ { r_{2}} ( \varphi_ {i }( z _ { n }) ) } \Psi(| f _ { w _ { n } }| ){\rm d} A \\
    & \lesssim \frac{\rho_{i,k} (z _ { n }) }{(1 - | w _ { n }| ^{2})^{2}} \int _ { \mathbb{D} } \Psi(| f _ { w _ { n } }| ){\rm d} A.
\end{align*}
Since $\|  f _ { w _ { n } } \| _ { A ^{\Psi}(\mathbb{D}) } \thickapprox 1$, Claim {\rm (iii)} implies that
\begin{align}\label{eq: T-ik}
( 1 - | w _ { n }|^{2}  ) ^{2} \Psi(| T_{i,k}f _ { w _ { n } } ( z_{n} ) |)  & \lesssim \rho_{i,k} (z _ { n })\to 0
\end{align}
for $i,k \in J _ { 1 }$.

Replacing $J$ and $\ell$ in \eqref{eq: operators-T} by $J_{1}$ and $k_{1}$, respectively, and using the convexity of $\Psi$ together with the $\Delta_{2}$ -condition, we deduce that
\begin{align*}
\Psi(|T _ { J_{1} }  f _ { w _ { n } } ( z _ { n } )|) &= \Psi(|\sum _ { i , k \in J_{1} } c _ { i ,k } T _ { i , k } f _ { w _ { n } } ( z _ { n } ) + \left( \sum _ { i \in J_{1} } a _ { i } \right) T _ { k_{1} } f _ { w _ { n } } ( z _ { n } )|)\\
     & \geq  \Psi\big( |\sum _ { i \in J_{1} } a _ { i } ||T _ { k_{1} } f _ { w _ { n } } ( z _ { n } )|-|\sum _ { i , k \in J_{1} }c _ { i ,k }||  T _ { i , k } f _ { w _ { n } } ( z _ { n } )|\big)  \\
    & \gtrsim \Psi\big( |\sum _ { i \in J_{1} } a _ { i } ||T _ { k_{1} } f _ { w _ { n } } ( z _ { n } )|\big)- \Psi\big(|\sum _ { i , k \in J_{1} }c _ { i ,k }||  T _ { i , k } f _ { w _ { n } } ( z _ { n } )|\big) .
\end{align*}
Since $\sum _ { i \in J_{1} } a _ { i }\neq0$, there exists a $\delta>0$ such that $\Big |\sum _ { i \in J_{1} } a _ { i } \Big|>\delta$. By \eqref{eq: remark}, there exists a constant $C_{\delta}>0$ such that
\begin{align*}
\Psi\big( |\sum _ { i \in J_{1} } a _ { i } ||T _ { k_{1} } f _ { w _ { n } } ( z _ { n } )|\big)
    & \geq C_{\delta}\Psi\big(|T _ { k_{1} } f _ { w _ { n } } ( z _ { n } )|\big) =C_{\delta}\frac{1}{( 1 - | w _ { n } | ^{2} ) ^{2}}.
\end{align*}
This yields
\begin{align*}
C_{\delta}\frac{1}{( 1 - | w _ { n } | ^{2} ) ^{2}} &\leq\Psi\big( |\sum _ { i \in J_{1} } a _ { i } ||T _ { k_{1} } f _ { w _ { n } } ( z _ { n } )|\big)\\
    &\lesssim \Psi(|T _ { J_{1} }  f _ { w _ { n } } ( z _ { n } )|)+\Psi\big(|\sum _ { i , k \in J_{1} }c _ { i ,k }||  T _ { i , k } f _ { w _ { n } } ( z _ { n } )|\big).
\end{align*}
Combining this with \eqref{eq: T-ik}, we obtain
\begin{align}\label{eq: T-J1}
  C_{\delta}  & \lesssim \lim_{n\to\infty} ( 1 - | w _ { n } | ^{2}) ^{2 }\Psi(|T _ { J_{1} }  f _ { w _ { n } } ( z _ { n } )|).
\end{align}

For $i\in \Lambda_{N}$, we get
\begin{align*}
    ( 1 - | w _ { n } | ^{2} ) ^{ 2} \Psi(| T _ { i } f _ { w _ { n } } ( z _ { n } ) |)
   & =( 1 - | w _ { n } | ^{2} ) ^{2} \Psi\bigg(\bigg|\Psi^{-1}\bigg(\frac{1}{( 1 - | w _ { n } | ^{2} ) ^{2}}\bigg)u_{w _ { n }}(\varphi_ {i }( z _ { n } ))\bigg|\bigg).
\end{align*}
In particular, when $i=k_{1}$, we have
 $$ ( 1 - | w _ { n } | ^{2} ) ^{ 2} \Psi(| T _ { k_{1} } f _ { w _ { n } } ( z _ { n } ) |)  =1.$$
 Since $|\varphi_ {i }( z _ { n } )|\leq 1$ for $i\in\Lambda_{N}$, we get $| 1 -  \overline { { w _ { n } } }\varphi _ { i } ( z _ { n } ) |\geq 1 - |\varphi _ { i } ( z _ { n } )|| w _ { n }|\geq1-| w _ { n } |\approx1-| w _ { n } |^{2}$. Thus $|u_{w _ { n }}(\varphi_ {i }( z _ { n } ))|=\left(\frac { 1 - | w _ { n } | ^{2} } { | 1 - \varphi _ { i } ( z _ { n } ) \overline { { w _ { n } } } |}\right) ^{s}\lesssim1$. By the convexity of $\Psi$, it follows that
\begin{align}\label{eq: T-i}
  ( 1 - | w _ { n } | ^{2} ) ^{ 2} \Psi(| T _ { i } f _ { w _ { n } } ( z _ { n } ) |)   & \lesssim( 1 - | w _ { n } | ^{2} ) ^{2}  \Psi\bigg(\Psi^{-1}\bigg(\frac{1}{( 1 - | w _ { n } | ^{2} ) ^{2}}\bigg)\bigg) |u_{w _ { n }}(\varphi_ {i }( z _ { n } ))|    \nonumber \\
   & =\bigg(\frac{1-|w _ { n } |^{2}}{|1-\overline{w _ { n } }\varphi _ { i } ( z _ { n } )|}\bigg)^{s} .
\end{align}

 Note that $| 1 -\overline { { w _ { n } } } \varphi _ { i } ( z _ { n } ) |\geq 1-|\varphi _ { i } ( z _ { n } )|$ for $i\in F$,
it follows from \eqref{eq: T-i} and the definition of $F$ that
\begin{align*}
  ( 1 - | w _ { n } | ^{2} ) ^{ 2} \Psi(| T _ { i } f _ { w _ { n } } ( z _ { n } ) |)
  & \lesssim \left( \frac { 1 - | w _ { n } | ^{2} } { 1 -  |\varphi _ { i } ( z _ { n } ) |^{2}  } \right) ^{ s} \\
   & \approx\left( \frac { 1 - | z _ { n } | ^{2} } { 1 -  |\varphi _ { i } ( z _ { n } ) |^{2}  } \right) ^{ s} = R _ {i } ^{s} ( z _ { n } ) \rightarrow  0.
\end{align*}
Applying \eqref{eq: T-J} and \eqref{eq: remark}, we conclude that
\begin{align*}
( 1 - | w _ { n } | ^{2} ) ^{ 2} \Psi(| T _ { F } f _ { w _ { n } } ( z _ { n } ) |) & =( 1 - | w _ { n } | ^{2} ) ^{ 2} \Psi(|  \sum _ { i \in F } a _ { i } T _ { i } f _ { w _ { n } } ( z _ { n } )|) \\
     & \lesssim \sum _ { i \in F } ( 1 - | w _ { n } | ^{2} ) ^{ 2} \Psi(| T _ { i } f _ { w _ { n } } ( z _ { n } ) |).
\end{align*}
Hence
\begin{align}\label{eq: T-F}
\lim_ { n \to \infty } ( 1 - | w _ { n } | ^{2} ) ^{ 2} \Psi(| T _ { F } f _ { w _ { n } } ( z _ { n } ) | ) = 0 .
\end{align}

Next, we consider the operators $T_{J_k}$, where $1 < k \leq L$. Claim (ii) yields $|\varphi_i(z_n)| \leq |w_n|$ for $i \in J_k$. So, by \eqref{eq: distance}, we have
\begin{align*}
  \left( \frac { 1 - | w _ { n } | ^{2} } { | 1 - \varphi _ { i } ( z _ { n } ) \overline { { w _ { n } } } | } \right) ^{2} & \leq \frac { ( 1 - | w _ { n } | ^{2} ) ( 1 - | \varphi _ { i } ( z _ { n } ) | ^{2} ) } { | 1 - \varphi _ { i } ( z _ { n } ) \overline { { w _ { n } } } | ^{2} }  \\
   & =\frac { ( 1 - | \varphi _ { k_{1} } ( z _ { n } ) | ^{2} ) ( 1 - | \varphi _ { i } ( z _ { n } ) | ^{2} ) } { | 1 - \varphi _ { i } ( z _ { n } ) \overline { {\varphi _ { k_{1}} ( z _ { n } ) } } | ^{2} }= 1 - \rho _ { k_{1} , i } ^{2} ( z _ { n } ).
\end{align*}
 From \eqref{eq: T-i}, we obtain
\begin{align*}
    ( 1 - | w _ { n } | ^{2} ) ^{ 2} \Psi(| T _ { i } f _ { w _ { n } } ( z _ { n } ) | )
   & \leq ( 1 - \rho _ {k_{1} , i } ^{2} ( z _ { n } ) )^{s/2}
\end{align*}
for $i \in J_k$.

Let $$\eta _ { 1 }  = \min\left[ \lim _ { n \to \infty } \rho _ { k _ { 1 } , i } ( z _ { n } ) \right] > 0,$$
where the minimum is taken over all $i \in \cup_{1 < k \leq L} J_k$.
 By \eqref{eq: remark} and \eqref{eq: T-J}, we have
\begin{align}\label{eq: T-Jk}
\limsup_{n\rightarrow\infty}( 1 - | w _ { n } | ^{2} ) ^{ 2} \sum_{1 < k \leq L} \Psi(| T _ { J _ { k } } f _ { w _ { n } } ( z _ { n } ) | ) &\lesssim \limsup_{n\rightarrow\infty} ( 1 - | w _ { n } | ^{2} ) ^{ 2} \Psi(|  T _ { i } f _ { w _ { n } } ( z _ { n } )|)  \nonumber \\
&\leq( 1 - \eta _ { 1 } ^{2} )^{s/2}.
\end{align}

Applying \eqref{eq: T-J} and Claim {\rm (i)}, we obtain  $$T _ { J _ { 1 } } = T - T _ { F } - \sum _ { 1 < k \leq L } T _ { J _ { k } } .$$
Moreover, \eqref{eq: subadditivity} and the convexity of $\Psi$ yield
\begin{align*}
( 1 - | w _ { n } | ^{2} ) ^{ 2} \Psi(| T _ { J _ { 1 } } f _ { w _ { n } } ( z _ { n } ) |) & =( 1 - | w _ { n } | ^{2} ) ^{ 2} \Psi(|  ( T - T _ { F } - \sum _ { 1 < k \leq L } T _ { J _ { k } } ) f _ { w _ { n } } ( z _ { n } )|) \\
       &\lesssim ( 1 - | w _ { n } | ^{2} ) ^{ 2} \bigg(\Psi(| Tf _ { w _ { n } } ( z _ { n } )|) +\Psi(|  T _ { F }f _ { w _ { n } } ( z _ { n } )|)  \\
       &+ \sum _ { 1 < k \leq L }\Psi(|  T _ { J _ { k } }  f _ { w _ { n } } ( z _ { n } )|)\bigg).
\end{align*}
It follows from \eqref{eq: lim-T}, \eqref{eq: T-J1}, \eqref{eq: T-F}, and \eqref{eq: T-Jk} that
$$C_{\delta} \lesssim( 1 - \eta _ { 1 } ^{2} )^{s/2} .$$
Taking the limit as $s \to \infty$, we arrive at
$$C_{\delta}=0.$$
This contradiction shows that
 $$\sum _ { i \in J _ { 1 } } a _ {i } = 0 .$$
Hence,  the proof is complete for the case $m=1$.

For the case $2 \leq m \leq L$, suppose that
$$\sum _ { i \in J _ { 1 } } a _ { i } = \cdots = \sum _ { i \in J _ { m - 1 } } a _ { i } = 0 .$$
It remains to show that
$$\sum _ { i \in J _ { m } } a _ { i } = 0. $$
Since the proof is  analogous to the case $m=1$, we omit the details. This completes the proof of the claim.

Based on the claim, we continue to verify the implication ${\rm (b)}  \Rightarrow {\rm (c)}$. One may assume that the sets $J_1,\ldots,J_L,F$ are all nonempty. Let $n_m$ and $d$ denote the number of elements in $J_m$ and $F$, respectively. Let
 $$\mathfrak  { j } _ { m }  = ( i _ { m , 1 } , \ldots , i _ { m , n _ { m } } ) ,  \qquad m = 1 , \ldots , L,$$
where $i_{m,1},\ldots,i_{m,n_m}$ are the distinct elements of $J_m$. Let $$\mathfrak { f }  = ( i _ { 1 } , \ldots , i _ { d } ),$$
where $i_{1},\ldots,i_{d}$ are the distinct elements of $F$. Constructing  $$\sigma = ( \mathfrak  { j } _ { 1 } , \ldots , \mathfrak { j } _ { L } , \mathfrak { f } )\in \mathscr{P}_N,$$
by \eqref{eq: define-Q}, we conclude that
 \begin{align*}
    Q _ { \sigma }  & = \sum _ { i = 1 } ^{N - 1} | s _ { i } ^{\sigma} | M _ { \sigma _ {i } ,\sigma _ { i + 1 } } + | s _ { N } ^{\sigma} | R _ { \sigma _ { N } } \\
    & = \bigg(\sum _ { i = 1 } ^{N - d} | s _ { i } ^{\sigma} | M _ { \sigma _ { i } , \sigma _ { i + 1 } }+\sum _ { i = N -d+1 } ^{N - 1} | s _ { i } ^{\sigma} | M _ { \sigma _ { i } , \sigma _ { i + 1 } }\bigg) + | s _ { N } ^{\sigma} | R _ { \sigma _ { N } } \\
    & =\sum _ { i = 1 } ^{N - d} | s _ { i } ^{\sigma} | M _ { \sigma _ { i } , \sigma _ { i + 1 } }+\bigg(\sum _ { i = N -d+1 } ^{N - 1} | s _ { i } ^{\sigma} | M _ { \sigma _ { i } , \sigma_ { i + 1 } } + | s _ { N } ^{\sigma} | R _ { \sigma _ { N } }\bigg).
 \end{align*}
 Applying Claim (iv) with $\sigma$, we deduce that
$$\sum _ { i = 1 } ^{n _ { 1} + \cdots + n _ { m } } a _ { \sigma _ { i } } = 0   \,\,\,\,\,\,\,   m=1,\ldots,L.$$
So
$$\sum _ { i = 1 } ^{N - d} \left| s _ { i } ^{\sigma} \right| M _ { \sigma _ { i } , \sigma _ { i + 1 } } = \sum _ { m = 1 } ^{L} \sum _ { i , k \in J _ { m } } c _ { i , k } ^{( m )} M _ { i , k },$$
where $c_{i,k}^{(m)}$ are nonnegative coefficients depending on $\{a_{\ell}:\ell\in J_{m}\}$.

On the other hand, noting that $M _ { \sigma _ { i } , \sigma _ { i + 1 } }=(R _ {\sigma_ { i } }+R _ { \sigma _ { i +1} })\rho_{\sigma_ { i } , \sigma _ { i + 1 }}$ and $0\leq\rho _ { \sigma _ { i } , \sigma _ { i + 1 } }<1$, so
 $$\sum _ { i = N - d + 1 } ^{N - 1} \left| s _ { i } ^{\sigma} \right| M _ {\sigma _ { i } , \sigma _ { i + 1 } } + \left| s _ { N } ^{\sigma} \right| R _ { \sigma _ { N } } \leq \sum _ { i \in F } c _ { i } R _ { i },$$
where $c_i$ are nonnegative coefficients depending on $\{a_{i_1},\ldots,a_{i_d}\}$.
Hence
$$Q _ { \sigma} (z_n)\leq\sum _ { m = 1 } ^{L} \sum _ { i , k \in J _ { m } } c _ { i , k } ^{( m )} M _ { i , k }(z_n) + \sum _ { i \in F } c _ { i } R _ { i }(z_n).$$
By Claim {\rm (iii)}, for $i,k\in J_{m}$, we have $\lim\limits _{n \rightarrow \infty} \rho_{i, k}(z_{n})=0$. From $M _ { i, k } ( z_{n} )  = [ R _ { i } ( z_{n} ) + R _ { k } ( z_{n} ) ] \rho _ { i , k } ( z_{n}) $, we obtain $\sum\limits _ { m = 1 } ^{L} \sum\limits _ { i , k \in J _ { m } } c _ { i , k } ^{( m )} M _ { i , k }(z_n)\rightarrow 0$.
Meanwhile, by \eqref{eq: define-F}, we conclude that $ \sum\limits _ { i \in F } c _ { i } R _ { i }(z_n)\rightarrow 0$.
So $$\lim _ { n \to \infty } Q _ { \sigma } ( z _ { n } ) = 0 .$$
This contradicts \eqref{eq: inf-Q}. Therefore  $\lim\limits_{|z|\to1}Q(z)=0$. This completes the proof.

Finally, we prove ${\rm (c)} \Rightarrow {\rm (a)}$.
Let $\{f_n\}$ be an arbitrary sequence in $A^{\Psi}(\mathbb{D})$ such that $\sup_{n} \|f_n\|_{A^{\Psi}(\mathbb{D})} \leq 1$ and $f_n\rightarrow0$ uniformly on compact subsets of  $\mathbb{D}$. By Lemma \ref{le: sequence-compact}, in order to prove $T$ is compact on $A^{\Psi}(\mathbb{D})$, it is enough to show that $\|Tf_n\|_{A^{\Psi}(\mathbb{D})}\rightarrow 0$. Moreover, by Lemma \ref{le: norm-convergence}, it suffices to prove that $\lim \limits_ { n \to \infty } \int _ { \mathbb{D} }\Psi ( | Tf _ { n }  | ) {\rm d}A = 0 $.

Fix $\varepsilon > 0$ and $\eta\in\mathscr{P}_N$, let
$$U _ {\eta, \varepsilon }  = \{ z \in \mathbb { D } : Q _ { \eta } (z) \leq \varepsilon \}.$$  Since $\lim\limits_{|z|\rightarrow1} Q(z)=0$, for any $\zeta\in\mathbb{T}$, there exists $h_\zeta=h_\zeta(\varepsilon)\in(0,1)$ such that
\begin{align}\label{eq: W-U}
W(\zeta,h_\zeta) \subset \bigcup _ { \eta \in \mathscr { P } _ { N } } U _ { \eta , \varepsilon } .
\end{align}
Otherwise, there would exist $\zeta_0 \in \mathbb{T}$ such that for any $\delta \in (0,1)$, the Carleson window $W(\zeta_0,\delta) $ contains a point $z_\delta\in \mathbb{D}$ satisfying  $Q_\eta(z_\delta) >\varepsilon$ for all $\eta \in \mathscr{P}_N$. Hence  $Q(z_\delta) > \varepsilon^{N!}$, which contradicts (c).

Since $\mathbb{T}$ is compact, we can find finitely many points $\zeta_{1},\ldots,\zeta_{j}\in \mathbb{T}$ such that
$$\mathbb{T}\subset  \bigcup _ { i = 1 } ^{j} B(\zeta_{i},h_{\zeta_{i}}),$$
where $B(\zeta_{i},h_{\zeta_{i}})$ denotes the Euclidean ball centered at $\zeta_{i}$ with radius $h_{\zeta_{i}}$. Set $r=\max_{i}(1-h_{i})$, where $h_{i}=h_{\zeta_{i}}$ and $0<h_{i}<1$. Then, for $ z\in\mathbb { D } \backslash r \mathbb { D }$, we have $|z|>r>1-h_{i}$. Thus, $z$ belongs to $W(\zeta _ { i },h_ { i } ) $ for some $1\leq i\leq j$. Hence $$\mathbb { D } \backslash r \mathbb { D } \subset \bigcup _ { i = 1 } ^{j} W(\zeta _ { i },h_ { i } ).$$
Combining this with \eqref{eq: W-U} yield
$$\mathbb { D } \backslash r \mathbb { D } \subset \bigcup _ { i = 1 } ^{j} W(\zeta _ { i },h_ { i } ) \subset \bigcup _ { \eta \in \mathscr { P } _ { N } } U _ { \eta , \varepsilon }.$$
Therefore
\begin{align*}
    \int _ { \mathbb{D} }\Psi ( | Tf _ { n }| )    {\rm d}A &=  \int _ { r\mathbb { D } } \Psi ( | Tf _ { n }  | )    {\rm d}A+ \int _ { \mathbb { D } \backslash r\mathbb { D }}\Psi ( | Tf _ { n }  | )    {\rm d}A  \\
  & \leq \int _ { r\mathbb { D } } \Psi ( | Tf _ { n }  | )    {\rm d}A+  \sum _ { \eta \in \mathscr { P } _ { N } }\int _ { U _ { \eta , \varepsilon }} \Psi ( | Tf _ { n }  | )    {\rm d}A\\
   & =I _ { 1 } + \sum _ { \eta \in \mathscr { P } _ { N } } I _ { 2 }.
\end{align*}

We now estimate $I _ { 1 }$ and $I _ { 2 }$, respectively.
Since $\bigcup_{i=1}^{N}\varphi_{i}(r\mathbb { D })$  is compact of $\mathbb { D }$,
and $\{f_{n}\}$ converges uniformly to 0 on $\bigcup_{i=1}^{N}\varphi_{i}(r\mathbb { D })$, we obtain
\begin{align}\label{eq: I-1}
I_{1}\rightarrow0 \quad  as  \quad  n\rightarrow\infty.
\end{align}
For $I _ { 2 } $,
by \eqref{eq: operator-T}, \eqref{eq: subadditivity} and the non-decreasing of $\Psi$, we arrive at
\begin{align*}
   I_{2} & \leq \int_{U_{\eta,\varepsilon}}\Psi\bigg(\sum_{i=1}^{N-1}|s_{i}^{\eta}||T_{\eta_{i},\eta_{i+1}}f _ { n } |+|s_{N}^{\eta}||T_{\eta_{N}}f _ { n } |\bigg) {\rm d}A \\
   & \lesssim\sum_{i=1}^{N-1}\int_{U_{\eta,\varepsilon}}\Psi\bigg(|s_{i}^{\eta}||T_{\eta_{i},\eta_{i+1}}f _ { n } |\bigg) {\rm d}A+\int_{U_{\eta,\varepsilon}}\Psi\bigg(|s_{N}^{\eta}||T_{\eta_{N}}f _ { n } |\bigg) {\rm d}A.
\end{align*}
Since $|s_{i}^{\eta}|>0$ for $1\leq i\leq N $, it follows from \eqref{eq: remark} that
\begin{align*}
    I_{2}&\lesssim\sum_{i=1}^{N-1}\int_{U_{\eta,\varepsilon}}\Psi\bigg(|T_{\eta_{i},\eta_{i+1}}f _ { n } |\bigg) {\rm d}A+\int_{U_{\eta,\varepsilon}}\Psi\bigg(|T_{\eta_{N}}f _ { n } |\bigg) {\rm d}A\\
   &=\sum_{i=1}^{N-1} \int_{U_{\eta,\varepsilon}} \Psi\bigg(\big|f_{n}\circ\varphi_{\eta_{i}}-f_{n}\circ\varphi_{\eta_{i+1}}\big|\bigg) {\rm d}A+\int_{U_{\eta,\varepsilon}}\Psi(|f_{n}\circ\varphi_{\eta_{N}}|) {\rm d}A\\
   &=\sum_{i=1}^{N-1} J_{1}+J_{2}.
\end{align*}

 We next estimate $J_{1}$ and $J_2$.
Since $Q_{\eta}(z)\leq\varepsilon$ for $ z\in U_{\eta,\varepsilon}$, it follows from \eqref{eq: define-Q} that $M_{\eta_{i},\eta_{i+1}}\lesssim\varepsilon$.
Let $K=U_{\eta,\varepsilon}$, we get $$\sup_{z\in K}M_{\eta_{i},\eta_{i+1}}\lesssim\varepsilon.$$ By Lemma \ref{le: symbols-difference}, there exists a constant
 $h_{i}(\varepsilon)>0$ such that $\lim\limits_{\varepsilon\rightarrow0}h_{i}(\varepsilon)=0$ and
\begin{align*}
    J_{1}& \leq  h_{i}(\varepsilon)\int_{\mathbb{D}}\Psi(|f_{n}|){\rm d}A\leq h_{i}(\varepsilon).
\end{align*}
For $J_{2}$, a similar argument gives $R_{\eta_{N}}\lesssim \varepsilon$. Set $H(z)=\chi_{U_{\eta,\varepsilon}}(z)$, then  $$\sup_{z\in\mathbb{D} }H(z)R_{\eta_{N}}(z)\leq\sup_{z\in U_{\eta,\varepsilon} }R_{\eta_{N}}(z)\lesssim\varepsilon.$$ By Lemma \ref{le: Orlicz-Carleson}, there exists a constant $C>0, \gamma=\frac{1}{2} $, independent of $\varepsilon$ such that
\begin{align*}
 J_{2}
   & = \int_{\mathbb{D} }\Psi(|f_{n}\circ\varphi_{\eta_{N}}|) H{\rm d}A \\
   & \leq  C \varepsilon^{\gamma}\int_{\mathbb { D }}\Psi(|f_{n}|) {\rm d}A  \leq  C \varepsilon^{\gamma}.
\end{align*}
Hence
\begin{align}\label{eq: I-2}
I _ { 2}  \leq\sum_{i=1}^{N-1}  h_{i}  ( \varepsilon ) + C\varepsilon ^{\gamma}.
\end{align}

It follows from \eqref{eq: I-1} and \eqref{eq: I-2} that
$$\limsup_{n\rightarrow\infty}\int_\mathbb { D }\Psi( |Tf_n|) {\rm d}A\leq\sum_{i=1}^{N-1}h_{i}  ( \varepsilon ) + C\varepsilon ^{\gamma}.$$
Taking the limit $\varepsilon\rightarrow0$, we obtain $\int_\mathbb { D }\Psi( |Tf_n|) {\rm d}A\rightarrow0$ as $n\rightarrow\infty.$
Therefore, $T$ is compact on $A^{\Psi}(\mathbb{D})$. This completes the proof of Theorem 1.1.


\section{Carleson measure characterization}
This section is devoted to the proof of Theorem 1.2. Recall from \eqref{eq: variable-change formula} that, for any positive Borel function $g$ on $\mathbb { D }$, the joint pullback measure is given by
\begin{align}\label{eq: measure}
  \int _ { \mathbb { D } } \Psi(|g| ) {\rm d} \mu _ {\Psi , \eta} & = \sum _ { i = 1 } ^{N-1} | s _ { i } ^{\eta} | \int _ {G _ { \eta } } \bigg[\Psi(|g\circ\varphi _ { \eta _ { i } } | )  +  \Psi(|g\circ\varphi _ { \eta _ { i+1 } } | )\bigg]\Psi( M_{\eta _ { i },\eta _ { i+1 }}) {\rm d}A \nonumber\\
   & +| s _ {N } ^{\eta} |\int _ { G _ { \eta } } \Psi(|g\circ\varphi _ { \eta_ { N } }| )\Psi( R _ { \eta _ { N } } )  {\rm d}A.
\end{align}

We now turn to the proof of the second main result.

\textbf{Proof of Theorem 1.2}.
Proof of ${\rm (b)}\Rightarrow{\rm (c)}$. Let $H=\chi_{G_{\eta}}$. Then,  for $1\leq i< N$ and any Borel set $E\subset\mathbb{D}$, we have
\begin{align*}
  \mu_{\Psi,\eta_i} (E)   & = \int _ { \varphi_{\eta _ { i }} ^{- 1} (E)\cap G_{\eta}}\Psi ( M _ { \eta _ { i },\eta _ { i+1 } } ) {\rm d} A +  \int _ { \varphi_{\eta_ { i+1 }} ^{- 1}(E)\cap G_{\eta} } \Psi ( M _ { \eta _ { i },\eta _ { i+1 } } ) {\rm d} A\\
  & =  \int _ { \varphi_{\eta _ { i }} ^{- 1} (E)}\Psi ( M _ { \eta _ { i },\eta _ { i+1 } } )H {\rm d} A +  \int _ { \varphi_{\eta_ { i+1 }} ^{- 1}(E)} \Psi ( M _ { \eta _ { i },\eta _ { i+1 } } )H {\rm d} A
\end{align*}
Since $\mu_{\Psi}$ is a vanishing Carleson measure on $\mathbb{D}$, it follows from  \eqref{eq: joint pullback measures} that
 $\mu_{\Psi,\eta_i}$ is a vanishing Carleson measure on $\mathbb{D}$.
Hence, by Lemma \ref{le: vanishing-Carleson-measure}, $\nu_{\Psi,\eta_i} $ is a vanishing Carleson measure on $\mathbb{D}$, where
\begin{align*}
  \nu_{\Psi,\eta_i} (E)   & = \int _ { \varphi_{\eta _ {i }} ^{- 1} (E)\cap G_{\eta}}  \Psi(\rho_{\eta _ { i },\eta _ { i +1}}) {\rm d} A +  \int _ { \varphi_{\eta _ { i+1 }} ^{- 1}(E)\cap G_{\eta} } \Psi(\rho_{\eta _ { i },\eta _ { i +1}}) {\rm d} A \\
    & = \int _ { \varphi_{\eta _ {i }} ^{- 1} (E)}  \Psi(\rho_{\eta _ { i },\eta _ { i +1}})H {\rm d} A +  \int _ { \varphi_{\eta _ { i+1 }} ^{- 1}(E)} \Psi(\rho_{\eta _ { i },\eta _ { i +1}})H {\rm d} A
\end{align*}
Therefore $$\nu_{\Psi}=\sum _ { \eta \in \mathscr { P } _ { N } } \nu _{\Psi,\eta} =\sum _ { \eta \in \mathscr { P } _ { N } }(\sum _ { i = 1 } ^{N} | s _ { i } ^{\eta} | \nu _ { \Psi,\eta _ { i } })$$ is a vanishing Carleson measure on $\mathbb{D}$.\\

Proof of ${\rm (c)}\Rightarrow{\rm (b)}$. Assume that $\nu_{\Psi}$ is a vanishing Carleson measure on $\mathbb{D}$. Note that
$$M_{\eta _ { i },\eta _ {i +1}}=(R _ { \eta _ { i } }+R _ { \eta _ { i +1}})\rho_{\eta _ { i },\eta_ { i +1}}$$ and  $R _ { i } (1\leq i< N)$ are bounded functions  on $\mathbb{D}$.
For $1\leq i< N$, by the non-decreasing of $\Psi$, we obtain
\begin{align*}
 \mu_{\Psi,\eta_i} (E)  &  = \int  _{\varphi _ { \eta _ { i } } ^{- 1} (E) \cap G _ { \eta } } \Psi(M_{\eta_ { i },\eta _ { i+1 }}){\rm d} A +\int  _{\varphi _ { \eta _ { i+1 } } ^{- 1} (E) \cap G _ { \eta } }\Psi( M_{\eta _ { i },\eta _ { i+1 }}) {\rm d} A \\
   &  \lesssim \int  _{\varphi _ { \eta _ { i } } ^{- 1} (E)\cap G _ { \eta } }\Psi(\rho_{\eta _ { i },\eta _ { i +1}}){\rm d} A +\int  _{\varphi _ { \eta _ { i+1 } } ^{- 1} (E) \cap G _ { \eta } }\Psi(\rho_{\eta _ { i },\eta _ { i +1}}) {\rm d} A  \\
   &  =\nu_{\Psi,\eta_i} (E)
\end{align*}
for any Borel set $E\subset\mathbb{D}$.
For $i=N$, we have $\mu_{\Psi,\eta_N} =\nu_{\Psi,\eta_N}$. Therefore,  $\mu_{\Psi} \lesssim\nu_{\Psi}$ for $1\leq i\leq N$.
Furthermore, we get $$\frac{\mu_{\Psi}(W(\xi,h))}{ h^{2}}\lesssim\frac{\nu_{\Psi}(W(\xi,h))}{h^{2}}.$$
Since $\nu_{\Psi}$ is a vanishing Carleson measure on $\mathbb{D}$,
it follows that $$\lim_{h\rightarrow0}\sup_{\xi\in\mathbb{T}}\frac{\mu_{\Psi}(W(\xi,h))}{ h^{2}}=0.$$
Hence, $\mu_{\Psi}$ is a vanishing Carleson measure on $\mathbb{D}$.

Proof of ${\rm (a)}\Rightarrow{\rm (b)}$.
It follows from \eqref{eq: joint pullback measures} that $\mu_{\Psi}$ is a vanishing Carleson measure on $\mathbb{D}$ if and only if $\mu_{\Psi,\eta}$ is a vanishing Carleson measure on $\mathbb{D}$. Thus, it suffices to show that $\lim\limits_{h\rightarrow0}\sup\limits_{\xi\in\mathbb{T}}\frac{\mu_{\Psi,\eta}(W(\xi,h))}{h^{2}}=0$.

By the definition of $\mu_{\Psi,\eta_{i}}$, for any Borel set $E\subset\mathbb{D}$, we obtain
\begin{align*}
  \mu _{\Psi,\eta} (E) =& \sum _ { i = 1 } ^{N} | s _ { i } ^{\eta} |  \mu _ { \Psi,\eta _ { i } } (E)=\sum _ { i = 1 } ^{N-1} | s _ { i } ^{\eta} |  \mu _ { \Psi,\eta _ { i } } (E)+| s _ { N } ^{\eta} |  \mu _ { \Psi,\eta _ { N} } (E) \\
   =& \sum _ { i = 1 } ^{N-1} \bigg(\int  _{\varphi _ { \eta _ { i } } ^{- 1} (E) \cap G _ { \eta } } | s _ { i } ^{\eta} |\Psi(M_{\eta _ { i },\eta _ { i+1 }}){\rm d} A +\int  _{\varphi _ { \eta _ { i+1 } } ^{- 1} (E) \cap G _ { \eta } }| s _ { i } ^{\eta} |\Psi( M_{\eta _ { i },\eta_ { i+1 }}) {\rm d} A\bigg)\\
   &+\int _ {   \varphi _ { \eta_ { N } } ^{- 1} (E) \cap G _ { \eta }} | s _ { N } ^{\eta} |\Psi( R _ { \eta _ { N } } )  {\rm d} A .
\end{align*}
Applying \eqref{eq: define-Q},
we have $Q _ { \eta}\geq| s _ { i } ^{\eta} |  M _ { \eta _ { i } , \eta _ { i + 1 } },\,  Q _ { \eta }\geq | s _ { N } ^{\eta} |  R _ { \eta _ { N } } $, i.e.,  $M _ { \eta _ { i } , \eta _ { i + 1 } }\leq\frac{Q _ {\eta }}{| s _ { i } ^{\eta} |}$,\, $R _ { \eta_ { N } }\leq\frac{Q _ { \eta }}{| s _ { N } ^{\eta} |}$.\\
Hence
\begin{align*}
  \mu _{\Psi,\eta} (E) &\lesssim\sum _ { i = 1 } ^{N-1} \bigg(\int  _{\varphi _ { \eta _ { i } } ^{- 1} (E)\cap G _ { \eta } } \Psi(Q _ { \eta }){\rm d} A +\int  _{\varphi _ { \eta _ { i+1 } } ^{- 1} (E) \cap G _ { \eta } }\Psi(Q _ { \eta}) {\rm d} A\bigg)+\int _ {   \varphi _ {\eta _ { N } } ^{- 1} (E) \cap G _ { \eta }} \Psi(Q _ { \eta })  {\rm d} A \\
  & \leq 2\sum _ {i = 1 } ^{N} \int _ { \varphi _ { i } ^{- 1} (E) \cap G _ { \eta } } \Psi(Q _ { \eta })   {\rm d} A.
\end{align*}
Set $E=W(\xi,h)$, we obtain
$$ \mu _{\Psi,\eta} (W(\xi,h))\lesssim2\sum _ { i = 1 } ^{N} \int _ { \varphi _ { i } ^{- 1} ( W(\xi,h)) \cap G _ {\eta} } \Psi(Q _ { \eta })   {\rm d} A.$$

Next, we estimate $\Psi(Q _ { \eta }) $ on $ \varphi _ { i } ^{- 1} ( W(\xi,h)) \cap G _ { \eta }$.
Let $\eta \in \mathscr{P}_{N}$. By \eqref{eq: define-G}, we have
$$Q _ { \eta } (z)\leq Q _ { \tau } (z)$$
for $ z \in G _ {\eta }$.
Taking the product over all $\tau\in \mathscr{P}_{N}$, we obtain
$$( Q _ { \eta } (z) ) ^{N !}=\Pi_{\tau\in \mathscr{P}_{N}}Q _ { \eta }(z)\leq \Pi_{\tau\in \mathscr{P}_{N}}Q _ { \tau } (z)=Q (z).$$
It follows from Theorem \ref{th: Julia-Caratheodory-type} that $T$ is compact on $A^{\Psi}(\mathbb{D})$ if and only if $\lim\limits_{|z|\to1}Q(z)=0$.
Thus, for any $\varepsilon > 0$, there exists $r \in (0,1)$ such that $Q (z)<\varepsilon^{N !}$ for $|z|>r$.
So
$$Q _ { \eta} (z) \leq ( Q (z) ) ^{\frac { 1} { N ! } } < \varepsilon,   \,\,\,\,\,\,\,z \in G _ { \eta } \cap \mathbb { D } ^{r}.$$
Choose $r_{0} \in (0,1)$  such that
$$\varphi _ { i } ( r \mathbb { D } ) \subset r _ { 0 } \mathbb { D } \,\,\,\,\,\text{or equivalently}\,\,\,\,\,\, \varphi _ { i } ^{- 1} ( \mathbb { D } ^{r _ { 0} } ) \subset \mathbb { D } ^{r}.$$
By the definition of the Carleson window  $W(\xi,h)$, there exists $\xi \in \mathbb{T}$ such that $W(\xi,h) \subset \mathbb{D}^{r_0}$.
Thus $$\varphi _ { i } ^{- 1}(W(\xi,h))\subset \varphi _ { i } ^{- 1}(\mathbb{D}^{r_0})\subset\mathbb { D } ^{r}.$$
Furthermore,
$$\varphi _ { i } ^{- 1}(W(\xi,h))\cap G _ { \eta }\subset\mathbb { D } ^{r}\cap G _ { \eta }.$$
Hence, for every $ z \in \varphi _ { i } ^{- 1}(W(\xi,h))\cap G _ { \eta}$, we have $ z \in \mathbb { D } ^{r}\cap G _ {\eta }$. So, $Q _ { \eta } (z)< \varepsilon$, and
\begin{align*}
    \frac{\mu _{\Psi,\eta} (W(\xi,h))}{h^{2}}  & \lesssim  2\Psi(\varepsilon)\sum _ { i= 1 } ^{N} \frac{\int _ { \varphi _ { i } ^{- 1} ( W(\xi,h)) }  {\rm d} A}{h^{2}}\\
   & =2\Psi(\varepsilon)\sum _ { i = 1 } ^{N} \frac{(A\circ \varphi _ { i } ^{- 1}) ( W(\xi,h)))}{h^{2}} .
\end{align*}
Since every composition operator is bounded on $A^{\Psi}(\mathbb{D})$, we have
 $$\lim_{h\rightarrow0}\sup_{\xi\in\mathbb{T}}\frac{\mu_{\Psi,\eta}(W(\xi,h))}{h^{2}}=0 .$$
Therefore, $\mu_{\Psi,\eta}$ is a vanishing  Carleson measure on $\mathbb{D}$,
which implies that $\mu_{\Psi}$ is also a vanishing  Carleson measure on $\mathbb{D}$.

Proof of ${\rm (b)}\Rightarrow{\rm (a)}$. To prove $T$ is compact on $A^{\Psi}(\mathbb{D})$, let $\{f_{n}\}$ be an arbitrary sequence in $A^{\Psi}(\mathbb{D})$ such that $\sup_{n}\|f_ {n}\|_{A^{\Psi}(\mathbb{D})} \leq1$ and $f_{n}\rightarrow0$ uniformly on compact subsets of $\mathbb{D}$. By Lemma \ref{le: sequence-compact}, it suffices to show that $\|Tf_ {n}\|_{A^{\Psi}(\mathbb{D})}\rightarrow0$.
Fix $r \in (0,1)$.
Since $r\mathbb{D}$ is a compact subset of $\mathbb{D}$,
$\{f_{n}\}$ converges uniformly to $0$ on $r\mathbb{D}$. Noting that
$$ \mathbb { D }= \bigcup _ { \eta \in \mathscr { P } _ { N } } G _ {\eta}, $$
we have
$$\mathbb{D}^{r}= \mathbb { D }\cap \mathbb{D}^{r} =\bigg(\bigcup _ { \eta \in \mathscr { P } _ { N } } G _ { \eta }\bigg)\cap \mathbb{D}^{r}= \bigcup _ { \eta \in \mathscr { P } _ { N } } \bigg(G _ { \eta }\cap \mathbb{D}^{r}\bigg)=\bigcup _ { \eta \in \mathscr { P } _ { N } } G_{\eta}^{r}.$$
Hence
\begin{align*}
 \limsup _ { n \to \infty } \int _ { \mathbb { D } } \Psi(| T f _ { n } |)    {\rm d} A & = \limsup _ { n \to \infty } \int _ {\mathbb { D }^{r} } \Psi(| T f _ { n } |)    {\rm d} A+\limsup _ { n \to \infty } \int _ { r\mathbb{D} }\Psi(| T f _ { n } |)   {\rm d} A \\
   & = \limsup _ { n \to \infty } \int _ { \mathbb { D }^{r} } \Psi(| T f _ { n } |)   {\rm d} A  \\
   & = \limsup _ { n \to \infty } \int _ { \bigcup _ { \eta \in \mathscr { P } _ { N } } G_{\eta}^{r} } \Psi(| T f _ { n } |)    {\rm d} A  \\
   & \leq \limsup _ { n \to \infty } \sum _ {\eta \in \mathscr { P } _ { N }}\int_{G _ { \eta } ^{r} }  \Psi(| T f _ { n } |)   {\rm d} A .
\end{align*}

Let $$T_{1}= \sum _ { i = 1 } ^{N - 1} s _ { i } ^{\eta} T _ { \eta _ { i } , \eta_ { i + 1 } },\,\,\,\,T_{2}=  s _ { N } ^{\eta} T _ { \eta _ {N}}. $$
By the non-decreasing of $\Psi$ and \eqref{eq: subadditivity}, we have
$$\Psi(|T_{1}f _ { n }|)\leq\Psi(\sum _ { i = 1 } ^{N - 1} |s _ { i } ^{\eta}| |T _ { \eta_ { i } ,\eta _ { i + 1 } } f _ { n }|)\lesssim\sum _ { i = 1 } ^{N - 1}\Psi( |s _ { i } ^{\eta}| |T _ { \eta_ { i } , \eta _ { i + 1 } } f _ { n }|).$$
For $0\leq|s _ { i } ^{\eta}|\leq1$, the convexity of $\Psi$ gives
$$\Psi( |s _ { i } ^{\eta}| |T _ { \eta _ { i } , \eta_ { i + 1 } } f _ { n }|)\leq|s _ { i } ^{\eta}|\Psi(  |T _ { \eta _ { i } , \eta_ { i + 1 } } f _ { n }|).$$
For $|s _ { i } ^{\eta}|>1$, the $\Delta_{2}$-condition yields
$$\Psi( |s _ { i } ^{\eta}| |T _ { \eta _ { i } , \eta _ { i + 1 } } f _ { n }|)\lesssim|s _ { i } ^{\eta}|^{p}\Psi(  |T _ { \eta _ { i } , \eta _ { i + 1 } } f _ { n }|)=|s _ { i } ^{\eta}|^{p-1} |s _ { i } ^{\eta}|\Psi(  |T _ { \eta _ { i } , \eta_ { i + 1 } } f _ { n }|).$$
Therefore
$$\Psi( |s _ { i } ^{\eta}| |T _ { \eta _ { i } , \eta_ { i + 1 } } f _ { n }|)\lesssim\max\{1,|s _ { i } ^{\eta}|^{p-1}\}|s _ { i } ^{\eta}|\Psi(  |T _ { \eta _ { i } , \eta _ { i + 1 } } f _ { n }|)\,\,\,\,\, \mbox{for}\,|s _ { i } ^{\eta}|\geq0, \,1\leq i< N.$$
Thus
$$\Psi(|T_{1}f _ { n }|)\lesssim\sum _ { i = 1 } ^{N - 1} |s _ { i } ^{\eta}|\Psi(  |T _ { \eta _ { i } , \eta _ { i + 1 } } f _ { n }|),$$
and
\begin{align*}
\int_{G _ { \eta } ^{r} }  \Psi(| T_{1} f _ { n } |)   {\rm d} A &\lesssim \sum _ { i = 1 } ^{N - 1} |s _ { i } ^{\eta}| \int_{G _ { \eta } ^{r} } \Psi(  |T _ { \eta_ { i } , \eta _ { i + 1 } } f _ { n }|) {\rm d} A \\
       & = \sum _ { i = 1 } ^{N - 1} |s _ { i } ^{\eta}|(I _ { 1n } ^{\eta, \varepsilon, i} + I _ { 2n } ^{\eta , \varepsilon , i} ),
\end{align*}
where
\begin{align}\label{eq: I-1n}
  I _ { 1n } ^{\eta, \varepsilon , i}  & =\int _ {G _ { \eta } ^{r} \cap \{ M _ { \eta _ { i } , \eta _ { i + 1 } } \geq \varepsilon\} }\Psi(|T _ { \eta _ { i } , \eta _ { i + 1 } }f _ { n } |)  {\rm d} A,
\end{align}
 \begin{align}\label{eq: I-2n}
 I _ { 2n } ^{\eta, \varepsilon , i} &=\int_{G _ {\eta } ^{r} \cap \{ M _ {\eta _ { i } , \eta _ { i + 1 } } < \varepsilon \}}  \Psi(| T _ { \eta _ { i } , \eta _ { i + 1 } } f _ { n } |)  {\rm d} A.
\end{align}
By the same argument, we  also obtain
$$\Psi(|T_{2}f _ { n }|)= \Psi(| s _ { N } ^{\eta} T _ { \eta _ {N}} f _ { n }|)\lesssim|s _ { N } ^{\eta}|\Psi(  |T _ { \eta _ { N } } f _ { n }|).$$
Hence
\begin{align*}
\int_{G _ { \eta } ^{r} }  \Psi(| T_{2} f _ { n } |)   {\rm d} A & \lesssim|s _ { N } ^{\eta}| \int_{G _ { \eta } ^{r} }  \Psi(  |T _ {\eta _ { N } } f _ { n }|)  {\rm d} A \\
 & = |s _ { N } ^{\eta}|(I _ { 1n } ^{\eta , \varepsilon, N} + I _ { 2n } ^{\eta, \varepsilon , N} ),
\end{align*}
where
\begin{align}\label{eq: I-1N}
  I _ { 1n } ^{\eta ,\varepsilon , N}  &=\int _ {G _ { \eta } ^{r} \cap \{ R _ { \eta _ { N }  } \geq \varepsilon\} }\Psi(|T _ { \eta _ { N } }f _ { n } |)  {\rm d} A,
\end{align}
\begin{align}\label{eq: I-2N}
I _ { 2n } ^{\eta , \varepsilon , N} &=\int_{G _ {\eta} ^{r} \cap \{ R _ {\eta _ { N } } < \varepsilon \}} \Psi(| T _ {\eta _ { N } } f _ { n } |)   {\rm d} A  .
\end{align}
Thus, for $T=T_{1}+T_{2}$, we deduce that
\begin{align*}
  \int_{G _ { \eta } ^{r} }  \Psi(| T f _ { n } |)   {\rm d} A   &\lesssim \sum _ { i = 1 } ^{N } |s _ { i } ^{\eta}| (I _ {1 n } ^{\eta, \varepsilon, i} + I _ {2 n } ^{\eta , \varepsilon , i} ) .
\end{align*}

To estimate \eqref{eq: I-1n}, it follows from the non-decreasing of $\Psi$ and \eqref{eq: subadditivity} that
\begin{align*}
  I _ { 1n } ^{\eta, \varepsilon , i} & =\int _ {G _ { \eta } ^{r} \cap \{ M _ { \eta_ { i } , \eta _ { i + 1 } } \geq \varepsilon\} }\Psi(|T _ { \eta _ { i } , \eta _ { i + 1 } }f _ { n } |)  {\rm d} A   \\
   & \leq \int _ {G _ { \eta } ^{r} \cap \{ M _ { \eta _ { i } , \eta_ { i + 1 } } \geq \varepsilon\} }\Psi(|T _ { \eta _ { i } , \eta _ { i + 1 } }f _ { n } |)  \frac{ \Psi(M _ { \eta_ { i } , \eta _ { i + 1 } }) }{\Psi(\varepsilon)}{\rm d} A \\
   & \leq \frac{ 1 }{\Psi(\varepsilon)}\int _ {G _ {\eta} }\Psi(|T _ { \eta_ { i } , \eta _ { i + 1 } }f _ { n } |)\Psi(M _ { \eta _ { i } ,\eta _ { i + 1 } })  {\rm d} A \\
   &=\frac{ 1 }{\Psi(\varepsilon)}\int _ {G _ {\eta } }\Psi(|(T _ { \eta_ { i }}- T_{\eta _ { i + 1 }})f _ { n } |)\Psi(M _ { \eta _ {i } , \eta _ { i + 1 } })  {\rm d} A\\
   &=\frac{ 1 }{\Psi(\varepsilon)}\int _ {G _ { \eta} }\Psi(|T _ { \eta_ { i }}f _ { n } - T _ {\eta _ { i + 1 } }f _ { n } |)\Psi(M _ { \eta _ { i } , \eta _ { i + 1 } })  {\rm d} A \\
   &\leq\frac{ 1 }{\Psi(\varepsilon)}\int _ {G _ { \eta } }\Psi(|T _ { \eta _ { i } }f _ { n }|+ |T _ {\eta _ { i + 1 } }f _ { n } |)\Psi(M _ { \eta _ { i } , \eta _ { i + 1 } })  {\rm d} A\\
   &\lesssim \frac{ 1 }{\Psi(\varepsilon)}\int _ {G _ { \eta } }[\Psi(|T _ {\eta _ { i } }f _ { n }|)+ \Psi(|T _ {\eta_ { i + 1 } }f _ { n } |)]\Psi(M _ { \eta _ { i } , \eta _ { i + 1 } })  {\rm d} A.
\end{align*}
For \eqref{eq: I-1N}, by the non-decreasing of $\Psi$, we get
\begin{align*}
  I _ { 1n } ^{\eta , \varepsilon , N} & =\int _ {G _ { \eta } ^{r} \cap \{R _ {\eta _ { N }  }\geq \varepsilon\} }\Psi(|T _ { \eta _ { N } }f _ { n } |)  {\rm d} A   \\
   & \leq \int _ {G _ { \eta } ^{r} \cap \{ R _ { \eta _ { N } } \geq \varepsilon\} }\Psi(|T _ { \eta _ { N }  }f _ { n } |)  \frac{ \Psi( R _ { \eta _ { N }  } ) }{\Psi(\varepsilon)} {\rm d} A \\
   & \leq \frac{ 1 }{\Psi(\varepsilon)}\int _ {G _ {\eta } }\Psi(|T _ { \eta_ { N }  }f _ { n } |)\Psi( R _ { \eta _ { N }  } )  {\rm d}A .
\end{align*}
Hence, for $ 1 \leq i \leq N$, it follows from \eqref{eq: measure} that
\begin{align*}
  \sum _ { i = 1 } ^{N } |s _ { i } ^{\eta}|I _ {1 n } ^{\eta, \varepsilon, i} & =\sum _ { i = 1 } ^{N-1 } |s _ { i } ^{\eta}|I _ { 1n } ^{\eta , \varepsilon, i} + |s _ { N } ^{\eta}|I _ { 1n } ^{\eta , \varepsilon, N} \\
   & \lesssim \frac{ 1 }{\Psi(\varepsilon)}\bigg(\sum _ { i = 1 } ^{N-1 } |s _ {i } ^{\eta}|\int _ {G _ { \eta } }[\Psi(|T _ { \eta_ { i } }f _ { n }|)+ \Psi(|T _ {\eta_ { i + 1 } }f _ { n } |)]\Psi(M _ { \eta _ { i } , \eta _ { i + 1 } })  {\rm d} A \\
   &+|s _ { N } ^{\eta}|\int _ {G _ { \eta } }\Psi(|T _ {\eta_ { N }  }f _ { n } |)\Psi( R _ { \eta _ { N }  } )  {\rm d} A \bigg)  \\
   & =\frac{ 1 }{\Psi(\varepsilon)}\int _ {\mathbb{D}}\Psi(|f _ { n } |)  {\rm d} \mu_{\Psi,\eta}.
\end{align*}
Since $\mu_{\Psi}$ is a vanishing Carleson measure on $\mathbb{D}$ and $\mu_{\Psi}= \sum _ { \eta \in \mathscr { P } _ { N } } \mu _{\Psi,\eta}$, it follows that $\mu_{\Psi,\eta}$ is a vanishing Carleson measure on $\mathbb{D}$. Moreover, since $\sup\limits_{n}\|f_ {n}\|_{A^{\Psi}(\mathbb{D})}
 \leq1$ and $f_{n}\rightarrow0$ uniformly on compact subsets of $\mathbb{D}$, Lemma \ref{le: norm-convergence} implies that $\lim\limits_{n\rightarrow\infty} \int _ {\mathbb{D}}\Psi(|f _ { n } |)  {\rm d} \mu_{\Psi,\eta}=0$.
Hence
$$\lim_{n\rightarrow\infty}\sum _ { i = 1 } ^{N } |s _ { i } ^{\eta}|I _ { 1n } ^{\eta , \varepsilon, i}=0.$$

For \eqref{eq: I-2n},
let $K=G _ { \eta } ^{r} \cap \{ M _ { \eta _ { i } ,\eta_ { i + 1 } } < \varepsilon \}\subset\mathbb{D}$, then we have $\sup\limits_{z\in K} M _ { \eta_ { i } ,\eta_ { i + 1 } } \leq\varepsilon$. By Lemma \ref{le: symbols-difference}, there exists a constant  $h_{\eta}(\varepsilon)=\sum _ { i = 1 } ^{N-1 } |s _ { i } ^{\eta}| h_{i}(\varepsilon) >0$ such that $\lim\limits_{\varepsilon\rightarrow0}h_{\eta}(\varepsilon)=0$ and
\begin{align*}
  \sum _ { i = 1 } ^{N-1 } |s _ { i } ^{\eta}| I _ { 2n } ^{\eta, \varepsilon , i} &=\sum _ { i = 1 } ^{N-1 } |s _ { i } ^{\eta}|\int_{K}  \Psi(| T _ { \eta _ { i } , \eta _ { i + 1 } } f _ { n } |)   {\rm d} A\\
   & = \sum _ { i = 1 } ^{N-1 } |s _ { i } ^{\eta}|\int_{K}\Psi(|T _ { \eta _ { i }}f _ { n } - T _ {\eta _ { i + 1 } }f _ { n } |) {\rm d}A \\
   & \leq \sum _ { i = 1 } ^{N-1 } |s _ { i } ^{\eta}|h_{i}(\varepsilon) \int_{\mathbb{D}}  \Psi(| f _ { n } |)   {\rm d} A  \leq h_{\eta}(\varepsilon) .
\end{align*}
For \eqref{eq: I-2N},
let $H=\chi_{G _ { \eta } ^{r} \cap \{ R _ { \eta_ { N } } < \varepsilon \}}\leq1$, then $\sup\limits_{z\in\mathbb{D}}H(z) R _ { \eta _ { N }}(z)\leq\varepsilon$. By Lemma \ref{le: Orlicz-Carleson}, there exists a constant $C=|s _ {N} ^{\eta}|C_{1}>0$ such that
\begin{align*}
 |s _ {N} ^{\eta}| I_ { 2n } ^{\eta, \varepsilon , N} & = |s _ {N} ^{\eta}| \int_{\mathbb{D}}  \Psi(|T _ { \eta_ { N }  } f _ { n } |) \chi_{G _ { \eta } ^{r} \cap \{ R _ {\eta _ { N } } < \varepsilon \}}  {\rm d} A\\
   & =  |s _ {N} ^{\eta}| \int_{\mathbb{D}}  \Psi(| T _ { \eta _ { N }  } f _ { n } |) H  {\rm d} A\\
   & \leq |s _ {N} ^{\eta}|C_{1}\varepsilon^{\gamma}\int_{\mathbb{D}} \Psi(| f _ { n } |)  {\rm d} A\leq C\varepsilon^{\gamma}.
\end{align*}

Therefore
\begin{align*}
  \limsup_{n\rightarrow\infty}\int_{G _ { \eta} ^{r} }  \Psi(| T f _ { n } |)   {\rm d} A &\leq \limsup_{n\rightarrow\infty}\sum _ { i= 1 } ^{N } |s _ { i } ^{\eta}| (I _ {1 n } ^{\eta , \varepsilon, i} + I  _ {2 n } ^{\eta, \varepsilon , i} ) \\
   & \leq  h_{\eta}(\varepsilon)+C\varepsilon^{\gamma}.
\end{align*}
Taking the limit $\varepsilon\rightarrow0$, we have
$$\limsup_{n\rightarrow\infty}\int_{G _ { \eta } ^{r} }  \Psi(| T f _ { n } |)   {\rm d} A=0.$$
 Hence $$\lim _ { n \to \infty } \int _ { \mathbb { D } } \Psi(| T f _ { n } |)    {\rm d} A=0.$$
It follows from Lemma \ref{le: norm-convergence}  that $\|Tf_{n}\|_{A^{\Psi}(\mathbb{D})}\rightarrow0$. This implies that $T$ is compact on $A^{\Psi}(\mathbb{D})$,
and the proof of Theorem 1.2 is complete.

\vspace{10mm}

\noindent\textbf{Data availability}\ Data sharing not applicable to this article as no datasets were generated or analysed during the current study.

\section*{Declarations}

\noindent\textbf{Conflict of interest}\ The authors have no relevant financial or non-financial interests to disclose.

\end{document}